\documentclass[preprint,12pt]{elsarticle}

\usepackage{amsmath, mathrsfs, amssymb, amsthm, mathtools, bbm} 

\usepackage{caption}

\usepackage{subfig}

\usepackage{makecell} 
\usepackage{float} 

\usepackage{setspace}

\usepackage{comment}

\usepackage[english]{babel} 
\usepackage[utf8]{inputenc} 
\usepackage[T1]{fontenc} 
\usepackage[margin=0.9in,footskip=0.25in]{geometry} 

\def \highlight{0}

\usepackage{tikz} 
\usetikzlibrary[topaths]

\usepackage{enumitem}

\usepackage{algorithm}
\usepackage{algcompatible} 

\usepackage{hyperref} 

\newtheorem{theorem}{Theorem}

\newtheorem{lemma}{Lemma}
\newtheorem{remark}{Remark}
\newtheorem{corollary}{Corollary}
\newtheorem{proposition}{Proposition}

\newtheorem*{theorem*}{Theorem}
\newtheorem*{example*}{Example} 
\newtheorem*{definition*}{Definition}
\newtheorem*{lemma*}{Lemma}
\newtheorem*{remark*}{Remark}
\newtheorem*{corollary*}{Corollary}
\newtheorem*{proposition*}{Proposition}
\newtheorem*{assumption*}{Assumption}
\newtheorem*{claim*}{Claim}

\newtheoremstyle{TheoremNum}
        {\topsep}{\topsep}              
        {\itshape}                      
        {}                              
        {\bfseries}                     
        {.}                             
        { }                             
        {\thmname{#1}\thmnote{ \bfseries #3}}
\theoremstyle{TheoremNum}

\newtheoremstyle{LemmaNum}
        {\topsep}{\topsep}              
        {\itshape}                      
        {}                              
        {\bfseries}                     
        {.}                             
        { }                             
        {\thmname{#1}\thmnote{ \bfseries #3}}
\theoremstyle{LemmaNum}

\newcommand{\<}{\left<}
\renewcommand{\>}{\right>}

\renewcommand{\[}{\left[} 
\renewcommand{\]}{\right]} 

\renewcommand{\(}{\left(}
\renewcommand{\)}{\right)}

\newcommand{\g}{ \mathbf{g} } 
\newcommand{\V}{ \mathbf{V} }

\renewcommand{\b}{ \mathbf{b} } 

\renewcommand{\cite}{\citep}

\renewcommand{\Pr}{ \mathbb{P} }

\newcommand{\prox}{\text{\textbf{prox}}}

\newcommand{\R}{\mathbb{R}}

\newcommand{\B}{\mathbb{B}}
\newcommand{\E}{\mathbb{E}}

\newcommand{\x}{\mathbf{x} }
\newcommand{\y}{\mathbf{y} }
\newcommand{\z}{\mathbf{z} }
\renewcommand{\v}{\mathbf{v} }

\usepackage{amssymb}
\usepackage{amsmath}

\journal{European Journal of Operational Research}

\begin{document}

\begin{frontmatter}



\title{The Nesterov--Spokoiny Acceleration Achieves Strict \MakeLowercase{\textit{o}(\textit{1/k}${}^2$)} Convergence}


\author[fdu1]{Weibin Peng} 
\ead{23210180060@m.fudan.edu.cn}

\author[fdu2]{Yu Liu} 
\ead{yuliu22@m.fudan.edu.cn}

\author[fdu2]{Tianyu Wang} 
\ead{wangtianyu@fudan.edu.cn}

\affiliation[fdu1]{organization={School of Mathematics, Fudan University},
            city={Shanghai},
            country={China}}

\affiliation[fdu2]{organization={Shanghai Center for Mathematica Sciences, Fudan University},
            city={Shanghai},
            country={China}}



\begin{abstract} 
    
    This paper studies the Nesterov-Spokoiny Acceleration (NSA), a variant of the accelerated gradient method by Nesterov and Spokoiny. For smooth convex optimization, NSA achieves a strict $o(1/k^2)$ convergence rate in function value and an $o(1/(k^3 \log k))$ rate in squared gradient norm, while ensuring monotonic descent of the objective. We further study a zeroth-order version of NSA that handles inexact gradients, and extends NSA to composite optimization problems, in each case establishing $o(1/k^2)$ convergence in function value. A continuous-time analysis reveals connections to high-resolution ODEs known to underlie acceleration phenomena. 
\end{abstract} 



\begin{keyword}
Continuous Optimization, Convex Optimization



\end{keyword}

\end{frontmatter}

\section{Introduction}



A central challenge in optimization is efficiently solving the problem 
\begin{align} 
    \min_{ \x \in \R^n}  f (\x) , \label{eq:goal}
\end{align}
where $f$ is a convex function bounded from below. 
\if\highlight1 
\color{red}
\fi 
The pursuit of robust solvers for this problem is driven by its profound real-world utility, as it encapsulates the core computational task in fields ranging from machine learning—where it facilitates model training via empirical risk minimization \cite{jordan2015machine}—to systems control \cite{lewis2012optimal}, signal processing \cite{orfanidis1995introduction}, and operations research \cite{pennanen2012introduction}. Consequently, advancements in convex optimization directly translate to enhanced performance and scalability across much of the modern technological landscape.

\if\highlight1 
\color{red}
\fi 
Given the widespread applicability of (1), developing efficient algorithms for its solution is of paramount importance. The most straightforward approach, gradient descent, often suffers from a slow convergence rate, particularly for ill-conditioned problems. Second-order methods, like Newton's method \citep[e.g.,][and references therein]{polyak2007newton}, address this by using the Hessian to achieve superlinear convergence, but at a prohibitive computational cost for large-scale problems and with only local convergence guarantees.

The quest to replicate this acceleration has therefore driven the development of advanced first-order methods. This pursuit is highly relevant to the rise of gradient-based zeroth-order optimization, which is particularly effective for fine-tuning large neural networks. In this context, the cost of computing gradients through back-propagation heavily outweighs the cost of a simple function evaluation via forward-propagation \citep{malladi2023finetuning}. Therefore, advances in efficient first-order methods potentially inform and inspire the development of better zeroth-order optimizers.

\if\highlight1 
\color{red}
\fi 

\subsection{Zeroth-order Optimization}

Recently, the rising demand for fine-tuning large language models has brought zeroth-order optimization back into the spotlight \cite[e.g.,][]{malladi2023finetuning,10.5555/3692070.3694514}. This resurgence is largely due to the memory efficiency of zeroth-order optimization, making it particularly suitable for large-scale deep learning challenges.

Although there is a vast body of literature on zeroth-order optimization (refer to \cite{conn2009introduction} for a contemporary overview), which encompasses techniques including Bayesian optimization \cite{shahriari2015taking}, stochastic optimization \cite{robbins1951stochastic,kiefer1952stochastic,nemirovski2009robust} and genetic algorithms \cite{srinivas1994genetic}, the application of zeroth-order optimization in fine-tuning is quite straightforward. We estimate the gradient using zeroth-order information and apply gradient-based methods using the estimated gradient. 

\begin{remark} 
    Over the years, gradient-based algorithms have become the default solver for (\ref{eq:goal}) in several scenarios, and a tremendous amount of gradient methods have been invented. The history of gradient algorithms dates back to classical times, when Cauchy proposed a method of finding the minimum value of a function via first-order information \cite{lemarechal2012cauchy}. Nowadays, gradient-based zeroth-order methods are gaining importance. 
\end{remark} 

Throughout the years, many gradient-based zeroth-order optimization methods have been studied (e.g.,\cite{flaxman2005online,ghadimi2013stochastic,duchi2015optimal,nesterov2017random,wang2018stochastic,liu2018zeroth,ji2019improved,balasubramanian2021zeroth,yue2023zeroth,Wang2023}); See \cite{liu2020primer} for an exposition. For example, \cite{flaxman2005online} studied the stochastic gradient estimator using a single-point function evaluation for the purpose of bandit learning. \cite{duchi2015optimal} studied stabilization of the stochastic gradient estimator via two-points (or multi-points) evaluations. \citet{nesterov2017random,balasubramanian2021zeroth} studied gradient/Hessian estimators using Gaussian smoothing, and investigated downstream optimization methods using the estimated gradient. 
In particular, \cite{nesterov2017random} introduced the first accelerated zeroth-order method that achieves an $O (k^{-2})$ convergence rate in expectation. 




\color{black}

\subsection{Accelerated Gradient Methods}   

\if\highlight1 
\color{red}
\fi 
Based on a common physics phenomenon and the original Gradient Algorithm, Polyak's Heavy-ball method was invented \cite{1964Some}. This algorithm is perhaps the first method that directly aims to accelerate the convergence of gradient search. 
One of the most surprising works in this field is the Nesterov's Accelerated Gradient (NAG) method \cite{nesterov1983method}, where accelerated gradient methods were invented to speed up the convergence for smooth convex programs. For such problems, NAG achieves a $O(k^{-2})$ convergence rate, without incurring computations of higher-order expense. Furthermore,
for bounded iteration counts $k$, the convergence rate of NAG is provably optimal among all methods whose iterates are constrained to the span of previous gradients and the initial point \cite{nest2018}.
\color{black}

Since Nesterov's original acceleration method \cite{nesterov1983method}, several accelerated gradient methods have been invented, including the one that inspires our work \cite{nesterov2011random,nesterov2017random}. Also, \cite{Flammarion2015FromAT} builds a connection between average gradient and acceleration. 
 Beyond these discrete-time analyses, the fundamental nature of the acceleration phenomenon has been fruitfully investigated from a continuous-time perspective. 
In particular, \cite{JMLR:v17:15-084} pointed out the relation between the a second-order ordinary differential equation and NAG. This discovery allows for a better understanding of Nesterov’s scheme. Afterwards, a family of schemes with $O(1/k^2)$ convergence rates has been obtained. \cite{2018Understanding} finds a finer correspondence between discrete-time iteration and continuous-time trajectory and proposes a gradient-correction scheme of NAG. \cite{chen2022gradient} introduced an implicit-velocity scheme of NAG and investigates on the convergence rate of gradient norms. 
Please See Section \ref{sec:ode} for more discussions on the continuous-time perspective. 



\if\highlight1 
\color{red}
\fi The principles of acceleration have been successfully combined with other powerful optimization tools, most notably the proximal operator; See e.g., \cite{parikh2014proximal} for an exposition. \color{black}
The proximal operator can handle possible nonsmoothness in the objective. 
Roughly speaking, the proximal operator finds the next iteration point by minimizing an approximation of the objective (or part of the objective) plus a regularizer term. In a broad sense, the mirror map in mirror descent \cite{blair1985problem} is also a type of proximal operator. In discussing related works, mirror descent algorithms are considered as proximal methods. 




Over the years, various accelerated gradient/proximal methods have been invented. \cite{nesterov2005smooth} introduced a proximal method that achieves $O(k^{-2})$ convergence rate for convex nonsmooth programs. 
\cite{2009A} popularized accelerated proximal methods via the Fast Iterative Shrinkage-Thresholding Algorithm (FISTA) algorithm, which achieves a better rate of convergence for linear inverse problems, while preserving the computational simplicity of ISTA \cite{combettes2005signal,daubechies2004iterative}. FISTA has also been analyzed in \cite{chambolle2015convergence}. 
\if\highlight1 
\color{red}
\fi \cite{Flammarion2015FromAT} shows that average gradient methods can also achieve acceleration. \citep{attouch2022first} introduced the Inertial Proximal Algorithm with Hessian Damping (IPAHD), an accelerated first-order method that utilizes the physics of Hessian damping. \color{black}
Motivated by the complementary performances of mirror descent and gradient descent, \cite{allen2014linear,doi:10.1137/18M1172314} couple these two celebrated algorithms and present a new scheme to reconstruct NAG on the basis of other algorithms. \cite{10.5555/2969442.2969558} builds a continuous-time analysis of the accelerated mirror descent method, whose discretized version also converges at rate $O(k^{-2})$. \cite{bubeck2015geometric} introduced a modification of NAG with a geometric interpretation and a proof of the same convergence rate as NAG. 
Notably, 
\if\highlight1 
\color{red}
\fi \cite{doi:10.1137/15M1046095} proved that the forward-backward method of Nesterov \cite{doi:10.1137/130910294} converges at rate $o(k^{-2})$.

However, one notable drawback of all Nesterov-type acceleration methods, when compared to gradient descent, is that the function value does not consistently decrease. This phenomenon arises from the momentum introduced during the gradient step. 
This paper presents a method that unifies the benefits of acceleration momentum and descent properties. While techniques for ensuring monotonicity in accelerated algorithms exist—such as the monotone APG by \cite{beck2009fast2} and subsequent APG-type methods by \cite{li2015accelerated} which ensure sufficient descent for nonconvex problems while maintaining an $O(k^{-2})$ rate for convex ones—their nonconvex convergence was later analyzed under the K{\L} framework by \cite{li2017convergence}. We are the first to demonstrate that such monotone methods simultaneously achieve a $ o(k^{-2}) $ $\limsup$ convergence in function value and $o(\frac{1}{k^3 \log k })$ $\liminf$ convergence in gradient norm. Furthermore, we prove that a zeroth-order implementation of our algorithm also enjoys a $o(k^{-2})$ convergence in function value.

\color{black}





\subsection{Our Contributions} 


In this paper, we study an acceleration gradient method called Nesterov--Spokoiny Acceleration (NSA), which is a variation of an acceleration method of Nesterov and Spokoiny \cite{nesterov2011random,nesterov2017random}. 
Our method harnesses the advantages of both accelerated gradient methods and traditional gradient descent -- it leverages an acceleration momentum, and ensures the decrease of function values. 
Theoretically, the NSA algorithm satisfies the following convergence rate properties: 
\begin{enumerate} 
     \item The sequence $\{ \x_k \}_{k \in \mathbb{N}} $ governed by NSA satisfies 
     \begin{align*}
         \limsup\limits_{k \to \infty } k^2 ( f (\x_k ) - f^* ) = 0 ,
     \end{align*}
     and $ \liminf\limits_{k \to \infty } k^2 \log k \log\log k ( f (\x_k ) - f^* ) = 0 $, 
     where $f^* > -\infty$ is the minimum of the smooth convex function $f$. 
    \item The sequence $\{ \y_k \}_{k \in \mathbb{N}} $ governed by NSA satisfies 
    \begin{align*}
        \liminf\limits_{k \to \infty } k^3 \log k \log\log k \| \nabla f (\y_k ) \|^2 = 0 . 
    \end{align*} 
\end{enumerate}

\if\highlight1 
\color{red}
\fi

To our knowledge, this work is the first to establish an $ o(k^{-2}) $ convergence rate for the function value and an $ o(\frac{1}{k^{3} \log k }) $ rate for the squared gradient norm, all while maintaining strictly monotonic descent of function value.






Also, this work presents a comprehensive study of the NSA algorithm. First, we introduce a zeroth-order variant capable of handling inexact gradients and prove its convergence at an $o(1/k^2)$ rate. Second, we propose a modification designed for convex nonsmooth problems, and provide a continuous-time analysis of its dynamics. Lastly, we adapt the framework to the composite optimization setting, where it likewise achieves an $o(1/k^2)$ convergence rate in the function value. 
A comparison of NSA with state-of-the-art methods is summarized in Table \ref{tab:full}.

\color{black}

\begin{table}[]
    \centering
    \begin{tabular}{|c|c|c|c|c|c|c|}
           \hline 
         & \makecell{$\limsup$ \\converg. \\rate for\\ func. value}
         & \makecell{$\liminf$ \\converg. rate\\ for grad. \\ norm sqr.} 
         & \makecell{grad.\\comp.\\per iter.} 
         & \makecell{guaranteed\\monot.\\descent?} 
         & \makecell{explicitly\\support\\prox.?} 
         \\ \hline \hline
         \cite{nesterov1983method} & $O (\frac{1}{k^2})$  & N/A & 1 & No &   N/A \\ \hline

        \cite{2009A} & $O (\frac{1}{k^2})$  & N/A & 1 & No & Yes  \\ \hline 
         \cite{allen2014linear} & $O (\frac{1}{k^2})$ &  N/A & 1 & No & Yes  \\ \hline 
         \cite{Flammarion2015FromAT} & $O (\frac{1}{k^2})$  & N/A& 1 & No & N/A   \\ \hline 
         
          \cite{10.5555/2969442.2969558} & $O (\frac{1}{k^2})$ & N/A  & 1 & No & Yes \\ \hline 
          \cite{JMLR:v17:15-084} & $O (\frac{1}{k^2})$  & N/A & 1 & No & Yes   \\ \hline
          \cite{doi:10.1137/15M1046095} & $o (\frac{1}{k^2})$ &  N/A  & 1 & No & Yes  \\ \hline 
         \cite{nesterov2017random} & $O (\frac{1}{k^2})$  & N/A & 1 & No & N/A  \\ \hline 
         \cite{doi:10.1137/18M1172314} & $O (\frac{1}{k^2})$  & N/A & 1 & No & Yes  \\ \hline 
        \cite{2019rate} & $o (\frac{1}{k^2})$ &  N/A & 1 & No & Yes  \\ \hline 
        \cite{attouch2022first} & $O (\frac{1}{k^2})$  & $O(\frac{1}{k^3})$ & 1 & No & Yes   \\ \hline 
        \cite{2018Understanding} & $o (\frac{1}{k^2})$  & $O(\frac{1}{k^3})$ & 1 & No & N/A \\ \hline 
          \cite{chen2022gradient} & $O (\frac{1}{k^2})$  & $o(\frac{1}{k^3})$ & 1 & No & N/A   \\
          \hline 
            \hline 
         \makecell{\textbf{This work}\\Theorem \ref{thm:main}, \ref{thm:prox-improve}} & $o (\frac{1}{k^2})$ & $o(\frac{1}{k^{3}\log k})$ & 2 & Yes & Yes  \\ 
        \hline 
    \end{tabular} 
    \caption{
    This table provides a chronological comparison of state-of-the-art acceleration methods for convex programs. The columns report key theoretical properties: the $\limsup$ convergence rate for the function value, the $\liminf$ convergence rate for the squared gradient norm, the number of gradient computations per iteration (excluding proximal operations), whether the algorithm ensures a monotonically decreasing objective value, and whether a proximal version has been explicitly analyzed. Entries marked "N/A" indicate that the corresponding result is not known or not applicable. Notably, our work is the first to \textbf{simultaneously} achieve $ o(k^{-2}) $ $\limsup$ convergence in function value, and $ o(\frac{1}{k^{3} \log k }) $ $\liminf$ convergence in squared gradient norm. 
    } 
    \label{tab:full} 
\end{table}

\section{The Nesterov--Spokoiny Acceleration} 


\if\highlight1 
\color{red}
\fi 
In the work of Nesterov and Spokoiny \cite{nesterov2017random}, an acceleration principle that differs from Nesterov's classic acceleration method \cite{nesterov1983method,nest2003} is presented. This method involves maintaining two sequences, $\{ \x_k \}$ and $\{ \z_k \}$, and utilizes the gradient on the line connecting $\x_k$ and $\z_k$ to update both trajectories. Nesterov and Spokoiny introduced this form of acceleration specifically in the context of zeroth-order acceleration.
To effectively enhance this method of Nesterov and Spokoiny, we integrate the principle of gradient descent into this acceleration method. A similar principle previously appeared in \cite{pmlr-v70-li17g}. In this paper,\textit{ we provide a design that merges the descent property with Nesterov's momentum, achieving a convergence rate of $o(k^{-2})$ while maintaining monotonic descent. }
\color{black}
The Nesterov--Spokoiny Acceleration algorithm is described in Algorithm \ref{alg:nsa}. 

\begin{algorithm}[H] 
    \caption{Nesterov--Spokoiny Acceleration (NSA)} 
    \label{alg:nsa} 
    \begin{algorithmic}[1]  
        \STATE \textbf{Initialize: }Pick $ \x_0 = \z_0 \in \R^n $ (such that $ \nabla f (\x_0) \neq 0 $.) 
        \FOR{$k = 0,1,2,\cdots,$}
            \STATE Compute $\y_{k} = (1 - \alpha_k) \x_k + \alpha_k \z_k $, with $\alpha_k = \frac{p}{k+p}$ for some constant $p$. 
            \STATE $\x_{k+1}' =  \y_k - \eta \nabla f (\y_k) $. /* $\eta$ is the step size. */
            \STATE $\x_{k+1}'' =  \x_k - \eta \nabla f (\x_k) $. 
            \STATE 
                $
                    \x_{k+1} =
                    \begin{cases}
                        \x_{k+1}', & \text{ if } f ( \x_{k+1}' ) \le f ( \x_{k+1}'' ) \\ 
                        \x_{k+1}'', & \text{ otherwise. } 
                    \end{cases} 
                $ 
            \STATE $\z_{k+1} = \z_k - \frac{ \eta }{ \alpha_k } \nabla f (\y_k) $. 
        \ENDFOR 
    \end{algorithmic} 
\end{algorithm}  

\begin{figure}[H]
     \centering
     {\includegraphics[width = 0.9\linewidth]{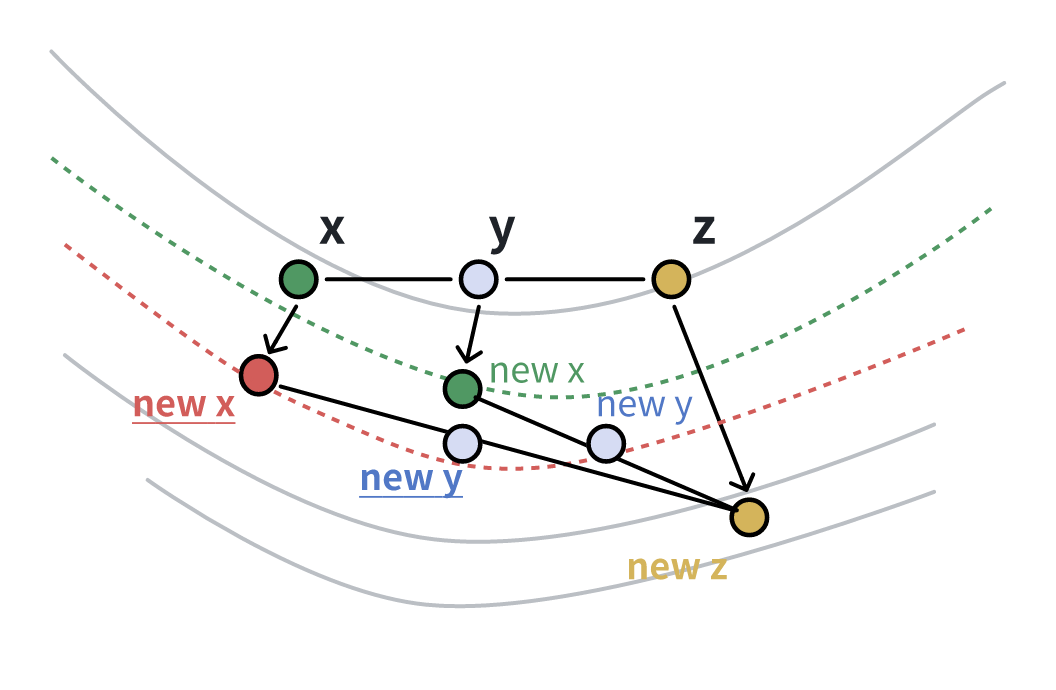}\caption{The figure above illustrates a comparison between Algorithm 1 and the original algorithm by Nesterov and Spokoiny \cite{2017Random}, which shows their paths in one iteration from the same starting point. The curves represent the contour lines. The points that are underlined are the new iteration points of Algorithm 1, while the rest new points are from the original algorithm by Nesterov and Spokoiny \cite{2017Random}. The red and green dashed lines represent the function values of the next step for $x_t$, respectively.} } \hfill
     
     \end{figure}


Convergence guarantee for smooth convex objectives $ \mathscr{F}_{L}^{1,1} (\R^n) $ \cite{nest2003,nest2018} is in Theorem \ref{thm:main}. For completeness, the definition of $ \mathscr{F}_{L}^{1,1} (\R^n) $ is included in the Appendix. 

\begin{theorem}
    \label{thm:main}
    Instate all notations in Algorithm \ref{alg:nsa}. Let  $p \ge 3$, and let $ \alpha_k = \frac{p}{k+p} $ for each $k \in \mathbb{N}$.
    Consider an objective function $f \in \mathscr{F}_{L}^{1,1} (\R^n) $. 
    \begin{enumerate}
        \item For any $\eta \in ( 0, \frac{2}{3L}]$, the NSA algorithm satisfies 
        \begin{align*}
            \limsup_{k \to \infty} k^2 (f (\x_k) - f^*) = 0 , 
            \quad \text{and} \quad 
            \liminf_{k \to \infty} k^2 \log k \log \log k (f (\x_k) - f^*) = 0 , 
        \end{align*}
        where $f^* = \min_{\x \in \R^n} f (\x) > -\infty$. 
        \item For any $\eta \in ( 0, \frac{2}{3L}]$, the NSA algorithm satisfies 
        \begin{align*} 
            \liminf_{k \to \infty} k^3 \log k \log \log k \| \nabla f (\y_k) \|^2 = 0 . 
        \end{align*} 
    \end{enumerate} 
    If the objective $f$ is $L$-smooth and satisfies $ \inf_{x \in \mathbb{R}^n} f (x) > - \infty $ ($f$ is possibly nonconvex), then the following holds. 
    \begin{itemize}
        \item For any $\eta \in ( 0, \frac{2}{3L}]$, the NSA algorithm satisfies 
        \if\highlight1 
        \color{red}
        \fi
        \begin{align*}
             {\min_{k \in \{ 0 , 1,\cdots, N \} }  \| \nabla f (\x_k) \|^2 
        \le 
        \frac{ 3 ( f (\x_0) - f^* ) }{ 2 \eta ( N+1 ) } } . 
        \end{align*}
        \color{black}
    \end{itemize}
\end{theorem} 

Next we prove Theorem \ref{thm:main}. We start the proof with Proposition \ref{prop:basic}, \if\highlight1 
\color{red}
\fi which derives the convergence rate of a nonnegative sequence based on its sum. \color{black}


\begin{proposition} 
    \label{prop:basic} 
    Let $\{ a_n \}_{n \in \mathbb{N}}$ be a nonnegative, nonincreasing sequence. If $ \sum_{n = 0}^\infty n a_n < \infty $, then $ \lim\limits_{n \to \infty } n^2 a_n = 0 $. 
\end{proposition} 

\begin{proof} 
    Since $ \{ a_n \}_{n \in \mathbb{N}} $ is nonnegative and $ \sum_{n = 0}^\infty n a_n < \infty $, we have 
    \begin{equation*} 
        \lim_{n \to \infty} \sum_{k=n}^{2n} k a_k = 0 . 
    \end{equation*} 
    Since $ a_n $ is nonincreasing, we have, for any $n$, $ \sum_{k=n}^{2n} k a_k \ge n^2 a_{2n} = \frac{1}{4} (2n)^2 a_{2n} $. This shows that $ \lim\limits_{n \to \infty} (2n)^2 a_{2n} = 0 $. 
    
    Similarly, we have $ \lim\limits_{n \to \infty } \sum_{k=n}^{2n+1} k a_k = 0 $. Repeating the above arguments shows that $ \lim\limits_{n \to \infty} (2n+1)^2 a_{2n+1} = 0 $. Thus $ \lim\limits_{n \to \infty} n^2 a_{n} = 0 $. 
\end{proof} 

\if\highlight1 
\color{red}
\fi 
Below we state and prove Lemma \ref{lem:decrease}, which describes the decrement property of the NSA algorithm. 
\color{black}

\begin{lemma} 
    \label{lem:decrease} 
    Let $f $ be $L$-smooth and let $ \eta \in (0, \frac{2}{3L}]$. Write $\delta_k := f (\x_k ) - f^*$ for simplicity. 
    Let $\x_k, \y_k$ be governed by the NSA algorithm. 
    Then for all $k = 0,1,2,\cdots$, it holds that
    \begin{itemize} 
        \item $ \delta_{k+1} \le \delta_k $, and more specifically $ f (\x_{k+1}) \le f (\x_k) - \frac{2\eta}{3} \| \nabla f (\x_k) \|^2 $; 
        \item $ f (\x_{k+1}) 
        \le 
        f (\y_k) - \frac{2\eta }{3} \| \nabla f (\y_k) \|^2 $. 
    \end{itemize} 
\end{lemma} 

\begin{proof}
    By $L$-smoothness of $f$, we know that 
    \begin{align} 
        f (\x_{k+1}') 
        \le& \;  
        f (\y_k) + \nabla f (\y_k)^\top \( \x_{k+1}' - \y_k \) + \frac{L}{2} \| \x_{k+1}' - \y_k \|^2 \nonumber \\ 
        \le & \; 
        f (\y_k) + \( - \eta  + \frac{1}{2} L \eta^2 \) \| \nabla f (\y_k) \|^2 . \nonumber 
    \end{align}  
    Since $ \eta \in (0, \frac{ 2 }{ 3L } ] $, we know 
    \begin{align}
        f (\x_{k+1}') 
        \le & \; 
        f (\y_k) - \frac{2\eta}{3} \| \nabla f (\y_k) \|^2 . \label{eq:x'} 
    \end{align} 

    Similarly, we have 
    \begin{align}
        f (\x_{k+1}'') 
        \le 
        f (\x_k) - \frac{2\eta}{3} \| \nabla f (\x_k) \|^2 . \label{eq:x''} 
    \end{align}
    By (\ref{eq:x''}), we know 
    \begin{align*}
        \delta_{k+1} 
        \le 
        f (\x_{k+1}'') - f (\x^*) 
        \le 
        f (\x_k) - \frac{2 \eta }{3} \| \nabla f (\x_k) \|^2 - f (\x^*) = \delta_k - \frac{2 \eta }{3} \| \nabla f (\x_k) \|^2 , 
    \end{align*} 
    which proves the first item. 
    By (\ref{eq:x'}) we know 
    \begin{align*}
        f (\x_{k+1}) 
        \le 
        f (\x_{k+1}' )
        \le 
        f (\y_k) - \frac{2\eta}{3} \| \nabla f (\y_k) \|^2 , 
    \end{align*} 
    which proves the second item. 
\end{proof} 

\if\highlight1 
\color{red}
\fi 
With Proposition \ref{prop:basic} and Lemma \ref{lem:decrease} in place, we are ready to prove Theorem \ref{thm:main}.
\color{black}

\begin{proof}[Proof of \if\highlight1 \color{red}\fi the first two items in \color{black} Theorem \ref{thm:main}] 
    Define 
    \begin{align*} 
        \Delta_k = \frac{\gamma_k}{2} \left\| \z_k - \x^* \right\|^2 + f (\x_k) - f (\x^*) , 
    \end{align*} 
    where $ \gamma_{k+1} = \frac{ \alpha_k^2 }{\eta} $ for all $ k = 0,1,\cdots $. 
    Thus for any $k \ge 1$, we have 
    \begin{align}
        \Delta_{k+1} 
        =& \;  
        \frac{\gamma_{k+1}}{2} \left\| \z_k - \frac{ \eta }{ \alpha_k} \nabla f (\y_k) - \x^* \right\|^2 + f (\x_{k+1}) - f (\x^*) \nonumber \\
        =& \; 
        \frac{\gamma_{k+1}}{2} \left\| \z_k - \x^* \right\|^2 - \< \frac{ \eta \gamma_{k+1} }{ \alpha_k} \nabla f (\y_k) , \z_k - \x^* \> + \frac{ \gamma_{k+1} \eta^2 }{ 2 \alpha_k^2} \| \nabla f (\y_k) \|^2 + f (\x_{k+1}) - f (\x^*) \nonumber \\
        =& \; 
        \frac{\gamma_{k+1}}{2} \left\| \z_k - \x^* \right\|^2 - \< \alpha_k \nabla f (\y_k) , \z_k - \x^* \> + \frac{ \eta }{2} \| \nabla f (\y_k) \|^2 + f (\x_{k+1}) - f (\x^*) \nonumber \\ 
        \le& \; 
        \frac{\gamma_{k+1}}{2} \left\| \z_k - \x^* \right\|^2 - \< \alpha_k \nabla f (\y_k) , \z_k - \x^* \> + f (\y_k) - \frac{\eta}{6} \| f (\y_k) \|^2 - f (\x^*) , \label{eq:pause-1}
    \end{align} 
    where the last inequality follows from Lemma \ref{lem:decrease}. 

    By the update rule, we have 
    \begin{align} 
        \Delta_{k+1}
        \le& \;  
        \frac{\gamma_{k+1}}{2} \left\| \z_k - \x^* \right\|^2 - \alpha_k  \< {\nabla} f (\y_k) , \z_k - \x^* \> +  f (\y_k)  - f (\x^*) - \frac{\eta}{6} \| \nabla f (\y_k) \|^2 \nonumber \\ 
        =& \; 
        \frac{\gamma_{k+1}}{2} \left\| \z_k - \x^* \right\|^2 + \< {\nabla} f (\y_k) , (1 - \alpha_k) \x_k + \alpha_k \x^* - \y_k \> + f (\y_k)   - f (\x^*) - \frac{\eta}{6} \| \nabla f (\y_k) \|^2 \nonumber \\ 
        \le& \; 
        \frac{\gamma_{k+1}}{2} \left\| \z_k - \x^* \right\|^2 + f ( (1 - \alpha_k) \x_k + \alpha_k \x^* )  - f (\x^*) - \frac{\eta}{6} \| \nabla f (\y_k) \|^2 \nonumber \\ 
        \le& \; 
        \frac{\gamma_{k+1}}{2} \left\| \z_k - \x^* \right\|^2 + (1 - \alpha_k) f (\x_k) + \alpha_k f (\x^*) -  f (\x^*) - \frac{\eta}{6} \| \nabla f (\y_k) \|^2 , \label{eq:pause-2}
    \end{align} 
    where the second line follows from the update rule $ \y_k = (1 - \alpha_k) \x_k + \alpha_k \z_k $, and the last two lines follow from convexity of $f$. 

    Rearranging terms in (\ref{eq:pause-2}) gives 
    \begin{align} 
        \frac{\Delta_{k+1} }{\gamma_{k+1}}
        \le& \;  
        \frac{\Delta_k }{\gamma_k} + \( \frac{1}{\gamma_{k+1}} - \frac{\alpha_k}{\gamma_{k+1}} - \frac{ 1 }{\gamma_k}  \) ( f (\x_k) - f (\x^*) ) - \frac{\eta}{6 \gamma_{k+1} } \| \nabla f (\y_k) \|^2 \label{eq:0} 
    \end{align} 
    
    Since $\alpha_k = \frac{p}{k+p}$ for some $p \ge 3$, we have that  
    \begin{align*} 
        \frac{1}{\gamma_{k+1}} - \frac{\alpha_k}{\gamma_{k+1}} - \frac{ 1 }{\gamma_k} 
        =& \;  
        \frac{ \eta (k+p)^2 }{ p^2} - \frac{\eta (k+p)}{p} - \frac{ \eta (k+p-1)^2 }{ p^2 } \\ 
        =& \;  
        \eta \frac{ \( 2 - p \) ( k+p ) - 1 }{  p^2 } \le - \eta \frac{k+p}{ p^2} . 
    \end{align*} 

    Again write $\delta_k := f (\x_k) - f (\x^*)$ for simplicity. Rearranging terms in (\ref{eq:0}) gives 
    \begin{align*}
        \eta \frac{k+p}{ p^2} \delta_k \le \frac{\Delta_k}{\gamma_k} - \frac{\Delta_{k+1}}{\gamma_{k+1}}, 
    \end{align*}
    and thus 
    \begin{align} 
        \sum_{k=1}^\infty \frac{k+p}{ p^2} \delta_k < \infty . \label{eq:sum}
    \end{align} 
    By Lemma \ref{lem:decrease}, we know that $ \{ \delta_k \}_{k=1}^\infty $ is decreasing. 
    Thus by Proposition \ref{prop:basic}, we know $  \limsup\limits_{k \to \infty} k^2 \delta_k = \lim\limits_{k \to \infty} k^2 \delta_k = 0 $. 

    Now assume, in order to get a contradiction, that there exists $\mu > 0$ such that $ \delta_k \ge \frac{\mu}{k^2 \log k \log \log k} $ for all $k \ge 3$. Then we have 
    \begin{align*} 
        \sum_{k=3}^\infty k \delta_k \ge \sum_{k=3}^\infty \frac{\mu}{k \log k \log \log k} = \infty , 
    \end{align*} 
    which is a contradiction to (\ref{eq:sum}). Therefore, no strictly positive $ \mu $ lower bounds $ \delta_k k^2 \log k \log \log k $ for all $k \ge 3$, meaning that 
    \begin{align} 
        \liminf_{k \to \infty} \delta_k k^2 \log k \log \log k = 0. \label{eq:liminf}
    \end{align} 
    Now we concludes the proof of item 1. 

    For item 2 of Theorem \ref{thm:main}, we again rearrange terms in (\ref{eq:0}) and take a summation to obtain 
    \begin{align*}
        \sum_{k=1}^\infty (k+p)^2  \| \nabla f (\y_k) \|^2 < \infty . 
    \end{align*} 
    An argument that leads to (\ref{eq:liminf}) finishes the proof of item 2 in Theorem \ref{thm:main}. 
    
\end{proof} 

\if\highlight1 
\color{red}
\fi 
\noindent \textbf{The Advantages of Ensuring Descending.} 
Beyond matching the $o(k^{-2})$ convergence rate of existing methods \citep{doi:10.1137/15M1046095}, the NSA method offers the distinct advantage of explicitly ensuring a descent property. This guarantees convergence to a local minimum even for nonconvex objectives, as shown below in Proposition \ref{prop:nonconv}. 

\begin{proposition} 
    \label{prop:nonconv} 
    Suppose the objective $f$ is nonconvex and $L$-smooth. If $\eta \in (0, \frac{2}{3L}]$, then it holds that 
    \begin{align*}
        \min_{k \in \{ 0 , 1,\cdots, N \} }  \| \nabla f (\x_k) \|^2 
        \le 
        \frac{ 3 ( f (\x_0) - f^* ) }{ 2\eta (N+1)} . 
    \end{align*}
\end{proposition} 

\begin{proof}
    By Lemma \ref{lem:decrease}, we have 
    \begin{align*}
        f (\x_{k+1}) \le f (\x_k) - \frac{2\eta}{3} \| \nabla f (\x_k) \|^2, 
    \end{align*}
    By telescoping, the above inequality implies 
    \begin{align*}
        \min_{k \in \{ 0 , 1,\cdots, N \} } \frac{2\eta}{3} \| \nabla f (\x_k) \|^2 \le \frac{2\eta}{3 (N+1)} \sum_{k=0}^N \| \nabla f (\x_k) \|^2 \le \frac{ f (\x_0) - f (\x_{N+1}) }{N+1} 
        \le 
        \frac{ f (\x_0) - f^* }{N+1}.  
    \end{align*} 
\end{proof} 

\color{black}

\section{The NSA Algorithm with Inexact Gradient Oracle} 
\label{sec:inexact}




\if\highlight1 
\color{red}
\fi 
Recently, there has been a growing demand for zeroth-order optimization techniques, primarily driven by the need for fine-tuning in contemporary deep learning (e.g., \cite{10.5555/3692070.3694514}). Additionally, many learning scenarios provide access solely to zeroth-order information of the objective function \cite{larson2019derivative}. 
\color{black}
In such cases, we can estimate the gradient first \cite{flaxman2005online,nesterov2017random,10.1093/imaiai/iaad014,10.1093/imaiai/iaac027} and then apply gradient algorithms to solve the problem. Such gradient estimators can be abstractly summarized as an inexact gradient oracle, defined as follows. 

For a function $f$, a point $\x \in \R^n$, estimation accuracy $\epsilon$, and a randomness descriptor $\xi$, the (inexact) gradient oracle $\mathcal{G}$ with parameters $f$, $\x$, $\epsilon$ (written $\mathcal{G} (f,\x,\epsilon; \xi)$) is a vector in $\R^n$ that satisfies: 
\begin{enumerate}[label=\textbf{(H\arabic*)}]
    \item For any $ \x \in \R^n$, any $\epsilon \in [0,\infty)$, there exists a constant $C$ that does not dependent on $\epsilon$, such that $\left\| \E_\xi \[ \mathcal{G} (f,\x,\epsilon; \xi) \] - \nabla f (\x) \right\| \le C \epsilon$. 
    \item At any $\x \in \R^n$ such that $ \nabla f (\x) \neq 0 $, there is $ \epsilon \in (0, \infty) $ such that 
    \begin{align*}
        \frac{3}{4} \| \mathcal{G} (f,\x,\epsilon; \xi) \|^2 \le \< \mathcal{G} (f,\x,\epsilon; \xi) , \nabla f (\x) \> \le \frac{5}{4} \| \mathcal{G} (f,\x,\epsilon; \xi) \|^2
    \end{align*} 
    almost surely. 
\end{enumerate}

Such gradient oracle $\mathcal{G}$ can be guaranteed by zeroth-order information. Formal statements for these properties are in Theorems \ref{thm:ge-a} and \ref{thm:ge-b}. Theorem \ref{thm:ge-a} has been proved in the literature by different authors \cite{flaxman2005online,nesterov2017random,10.1093/imaiai/iaad014}. 
\if\highlight1 
\color{red}
\fi 
The key message is that inexact gradient oracles that satisfy conditions \textbf{(H1)} and \textbf{(H2)} can be derived using zeroth-order information. 
A simple method that ensures conditions \textbf{(H1)} and \textbf{(H2)} is performing finite differences along each coordinate to derive the gradient estimator. Specifically, we consider the most straightforward finite difference gradient estimator defined as follows: 
\begin{align} 
    \mathcal{G}(f, \x, \epsilon;\xi) := \frac{1}{2\epsilon }\sum_{k=1}^{n}\[ f(x+\epsilon\mathbf{e}_i )-f(x-\epsilon \mathbf{e_i})\]\mathbf{e}_i, \quad \forall \x \in \R^n, \label{eq:def-grad-est}
\end{align} 
where $ \mathbf{e}_i $ is the $i$-th unit vector in a fixed canonical basis. 
The straightforward estimator in (\ref{eq:def-grad-est}) serves as an example that fulfills conditions \textbf{(H1)} and \textbf{(H2)}. It is important to note that this estimator is entirely deterministic, making $\xi $ a mere placeholder, and thus $ \E_\xi $ acts as an identity operator. However, (\ref{eq:def-grad-est}) is not the only simple estimator that meets these criteria; the estimator presented in \cite{10.1093/imaiai/iaad014} also satisfies conditions \textbf{(H1)} and \textbf{(H2)}.


\color{black}

\begin{theorem}[\cite{flaxman2005online,nesterov2017random,10.1093/imaiai/iaad014}]
    \label{thm:ge-a}
    Let the function $f$ be $L$-smooth. If at any $\x \in \R^n$, we can access the function value $f (\x)$, then there exists an oracle $\mathcal{G}$ such that (1) $\mathcal{G}$ satisfies property \textbf{(H1)}, and (2) $\mathcal{G}$ can be explicitly calculated using zeroth-order information. 
\end{theorem}

\begin{theorem} 
    \label{thm:ge-b} 
    Consider a function $f \in \mathscr{F}$, where $ \mathscr{F} $ denotes Nesterov's $\mathscr{F}$ functions \cite{nest2003,nest2018}. If at any $\x \in \R^n$, we can access the function value $f (\x)$, then there exists an oracle $ \mathcal{G} $ such that (1) $ \mathcal{G} $ satisfies property \textbf{(H2)}, and (2) $ \mathcal{G} $ can be explicitly calculated using zeroth-order information. 
\end{theorem}  

\if\highlight1 
\color{red}
\fi 
We now present a version of NSA that operates with the inexact gradient oracle, derived from zeroth-order information. The complete procedure is outlined in Algorithm \ref{alg:nsa-inexact}.

\color{black}


\begin{algorithm}[H] 
    \caption{Nesterov--Spokoiny Acceleration (NSA) that handles inexact gradient} 
    \label{alg:nsa-inexact} 
    \begin{algorithmic}[1] 
        \STATE \textbf{Inputs:} Gradient oracle $\mathcal{G}$. /* The gradient oracle $\mathcal{G}$ needs not be exact. */ 
        \STATE \textbf{Initialize: }Pick $ \x_0 = \z_0 \in \R^n $ (such that $ \nabla f (\x_0) \neq 0 $.) 
        \FOR{$t = 0,1,2,\cdots,$}
            \STATE Compute $\y_{k} = (1 - \alpha_k) \x_k + \alpha_k \z_k $, with $\alpha_k = \frac{p}{k+p}$ for some constant $p$. 
            \STATE Obtain possibly inexact and stochastic gradients $ \g_k (\x_k) := \mathcal{G} (f, \x_k, \epsilon_k; \xi_k) $ and $ \g_k (\y_k) := \mathcal{G} (f, \y_k, \epsilon_k; \xi_k') $. 
            \STATE /* $\epsilon_k$ is the error of the gradient oracle at step $k$. $\xi_k$, $\xi_k'$ describe possible randomness in the gradient oracle. If the gradient oracle is exact, then $\epsilon_k = 0$ and $ \mathcal{G} (f, \cdot, 0; \xi) = \nabla f (\cdot )$ surely. */ 
            \STATE $\x_{k+1}' =  \y_k - 2 \eta \g_k (\y_k) $. /* $\eta$ is the step size. */
            \STATE $\x_{k+1}'' =  \x_k - 2 \eta \g_k (\x_k) $. 
            \STATE 
                $
                    \x_{k+1} =
                    \begin{cases}
                        \x_{k+1}', & \text{ if } f ( \x_{k+1}' ) \le f ( \x_{k+1}'' ) \\ 
                        \x_{k+1}'', & \text{ otherwise. } 
                    \end{cases} 
                $ 
            \STATE $\z_{k+1}' = \z_k - \frac{ \eta }{ \alpha_k } \g_k (\y_k) $. 
            \STATE 
                $
                \begin{cases}
                    \text{\textbf{Option I: }}\z_{k+1} = \z_{k+1}'. \\
                    \text{\textbf{Option II: }} \z_{k+1} = \Pi_{\B_\mu} (\z_{k+1}' ), \text{where $ \Pi_{\B_\mu} $ is the projection to the ball of radius $\mu$}. 
                \end{cases} 
                $ 
            \STATE \quad /* Option II is needed only when the gradient oracle is inexact. */ 
        \ENDFOR 
    \end{algorithmic} 
\end{algorithm}  

\if\highlight1 
\color{red}
\fi 
The performance guarantee of Algorithm \ref{alg:nsa-inexact} is summarized below in Theorem \ref{thm:main-inexact}. 
\color{black}

\begin{theorem}
    \label{thm:main-inexact}
    Instate all notations in Algorithm \ref{alg:nsa-inexact}. 
    Consider an objective function $f \in \mathscr{F}_{L}^{1,1} (\R^n) $. Let $\eta \in ( 0, \frac{1}{2L}]$ and fix $p \ge 3$, and let $ \alpha_k = \frac{p}{k+p} $ for each $t \in \mathbb{N}$. 
    In cases where the gradient oracle $ \mathcal{G} $ is inexact and possibly stochastic, then if (a) the gradient oracle satisfies properties \textbf{(H1)} and \textbf{(H2)}, (b) the gradient error $\epsilon_k$ satisfies $ \sum_{k=1 }^\infty \frac{\epsilon_k}{\alpha_k} < \infty $, and (c) the minimizer $\x^*$ of $f$ satisfies $\x^* \in \mathbb{B}_\mu$, then the NSA algorithm with Option II satisfies 
    \begin{align} 
        \limsup_{k \to \infty} k^2 (f (\x_k) - f(\x^*)) = 0 \;\;\; a.s. , \label{eq:thm-inexact-1}
    \end{align} 
    and
    \begin{align}
        \liminf_{k \to \infty} k^2 \log k \log \log k (f (\x_k) - f (\x^*) ) = 0 , \;\;\; a.s. \label{eq:thm-inexact-2}
    \end{align} 
\end{theorem} 

\if\highlight1 
\color{red}
\fi 

In Theorem \ref{thm:main-inexact}, conditions on the gradient error oracles are imposed. Below in Corollary \ref{cor}, we present a guarantee that removes these conditions by using a closed-form gradient estimator. 

\begin{corollary} 
    \label{cor} 
    Instate all notations in Algorithm \ref{alg:nsa-inexact}. 
    Consider an objective function $f \in \mathscr{F}_{L}^{1,1} (\R^n) $. Let $\eta \in ( 0, \frac{1}{2L}]$ and fix $p \ge 3$, and let $ \alpha_k = \frac{p}{k+p} $ for each $t \in \mathbb{N}$. Let $ \epsilon_k = 2^{-k} $, and let $\mathcal{G}$ be defined as in (\ref{eq:def-grad-est}). 
    Then if the minimizer $\x^*$ of $f$ satisfies $\x^* \in \mathbb{B}_\mu$, then the NSA algorithm with Option II satisfies either 
    \begin{align}
        \liminf_{k \to \infty } \beta^k \| \nabla f (\x_k) \| = 0 \; \text{ for some $\beta > 1$}, \label{eq:cor-1}
    \end{align}
    or 
    \begin{align*} 
        \limsup_{k \to \infty} k^2 (f (\x_k) - f(\x^*)) = 0 , \; \; \text{and} \;\; \liminf_{k \to \infty} k^2 \log k \log \log k (f (\x_k) - f (\x^*) ) = 0 . 
    \end{align*} 
\end{corollary}  

Corollary \ref{cor} follows directly from Theorem \ref{thm:main-inexact} and a refined application of Theorem \ref{thm:ge-b}; the proof can be found in the Appendix.
Next we prove Theorem \ref{thm:main-inexact}. The proof starts with Lemma \ref{lem:decrease-inexact}, which establishes the decent property of Algorithm \ref{lem:decrease-inexact}.

\color{black}




\begin{lemma}
    \label{lem:decrease-inexact}
    Let $f \in \mathscr{F}_L^{1,1} (\R^n)$ and let $ \eta \in (0, \frac{3}{4 L}]$. Let $ \alpha_k = \frac{p}{k+p}$ for some $p > 0$. Let $\x_k$ be governed by the NSA algorithm (Algorithm \ref{alg:nsa-inexact}). 
    If the gradient estimation oracle $\mathcal{G}$ satisfies properties \textbf{(H1)} and \textbf{(H2)}, 
    then there exists an error sequence $\{ \epsilon_k \}_{k\in\mathbb{N}}$, such that $ \sum_{k=1}^\infty \frac{\epsilon_k}{\alpha_k} < \infty $, and for all $k = 0,1,2,\cdots$, it holds that 
    \begin{align*}
        \delta_{k+1} \le \delta_k \quad \text{ and } \quad f (\x_{k+1}) 
        \le 
        f (\y_k) - \eta \| \g_k (\y_k) \|^2 , 
    \end{align*}
    where $\delta_k := f (\x_k ) - f(\x^*)$ is introduced to avoid notation clutter. 
\end{lemma}

\begin{proof}
    Since the gradient estimation oracle $ \mathcal{G} $ satisfies property \textbf{(H2)}, we know that, at any $\y_k $ with $\nabla f (\y_k) \neq 0$, there exists $\epsilon' > 0$ such that 
    \begin{align}
        \frac{3}{4} \| \mathcal{G} (f, \y_k, \epsilon'; \xi_k' )  \|^2 
        \le 
        \nabla f (\y_k)^\top \mathcal{G} (f, \y_k, \epsilon'; \xi_k' )
        \le 
        \frac{5}{4} \| \mathcal{G} (f, \y_k, \epsilon'; \xi_k' ) \|^2 , \label{eq:R-grad-f}
    \end{align}
    almost surely. By letting $ \epsilon_k $ to be smaller than this $\epsilon'$, we have 
    \begin{align*}
        \frac{3}{4} \| \g_{k} (\y_k)  \|^2 
        \le 
        \nabla f (\y_k)^\top \g_{k} (\y_k) 
        \le 
        \frac{5}{4} \| \g_{k} (\y_k) \|^2 , \quad a.s.
    \end{align*}
    Note that we can find a $\epsilon_k \le \epsilon' $ at any $k$ without violating $ \sum_{k=1}^\infty \frac{\epsilon_k}{\alpha_k} < \infty $. 
    
    By $L$-smoothness of $f$, we know that 
    \begin{align} 
        f (\x_{k+1}') 
        \le& \;  
        f (\y_k) + \nabla f (\y_k)^\top \( \x_{k+1}' - \y_k \) + \frac{L}{2} \| \x_{k+1}' - \y_k \|^2 \nonumber \\ 
        =& \;  
        f (\y_k) - 2 \eta \nabla f (\y_k)^\top \g_k (\y_k) + 2 L \eta^2 \| \g_k (\y_k) \|^2 \nonumber \\ 
        \le & \; 
        f (\y_k) + \( - \frac{5}{2} \eta + 2 L \eta^2 \) \| \g_k (\y_k) \|^2 , \nonumber 
    \end{align}  
    where the last inequality follows from (\ref{eq:R-grad-f}).  Since $ \eta \in (0, \frac{3}{4 L } ] $, we know 
    \begin{align}
        f (\x_{k+1}') 
        \le & \; 
        f (\y_k) - \eta  \| \g_k (\y_k) \|^2 . \label{eq:x'-inexact} 
    \end{align} 

    Similarly, we have 
    \begin{align}
        f (\x_{k+1}'') 
        \le 
        f (\x_k) + \nabla f (\x_k)^\top \( \x_{k+1}'' - \x_k \) + \frac{L}{2} \| \x_{k+1}'' - \x_k \|^2 
        \le 
        f (\x_k) - \eta \| \g_k (\x_k) \|^2 . \label{eq:x''-inexact} 
    \end{align}
    By (\ref{eq:x''-inexact}), we know 
    \begin{align*}
        \delta_{k+1} 
        \le 
        f (\x_{k+1}'') - f (\x^*) 
        \le 
        f (\x_k) - f (\x^*) = \delta_k . 
    \end{align*} 

    By (\ref{eq:x'-inexact}) we know 
    \begin{align*}
        f (\x_{k+1}) 
        \le 
        f (\x_{k+1}' )
        \le 
        f (\y_k) - \eta \| \g_k (\y_k) \|^2 . 
    \end{align*} 
\end{proof}

    


\if\highlight1 
\color{red}
\fi 
In Lemma \ref{lem:key}, we prove that the optimality gap can be telescoped. This result allows us to readily apply Proposition \ref{prop:basic}. \color{black}
\begin{lemma}
    \label{lem:key}
    Instate all notations in Algorithm \ref{alg:nsa-inexact}. 
    Consider an objective function $f \in \mathscr{F}_{L}^{1,1} (\R^n) $. 
    Let $\alpha_k \in (0,1]$ for all $k \in \mathbb{N}$. 
    Define $\gamma_{k+1} := \frac{\alpha_k^2}{\eta}$ for all $k \ge 1$, and define 
    \begin{align*}
        \Delta_{k} := \frac{\gamma_{k}}{2} \| \z_k - \x^* \|^2 + \delta_k, 
    \end{align*} 
    where $\delta_k := f (\x_k ) - f^*$. 
    For Algorithm \ref{alg:nsa-inexact} with Option II, if $ \x^* \in \mathbb{B}_\mu $, then there exists a constant $C$ such that 
    \begin{align*}
        \E \[ \( \frac{1}{\gamma_{k}} + \frac{\alpha_k}{\gamma_{k+1}} - \frac{1}{\gamma_{k+1}} \) \delta_k \] 
        \le 
        \E \[ \frac{\Delta_{k}}{\gamma_{k}} \] - \E \[ \frac{\Delta_{k+1}}{\gamma_{k+1}} \] + \frac{ 2 C \mu \alpha_k  \epsilon_k }{ \gamma_{k+1} } 
    \end{align*} 
    for all $k  = 1,2,\cdots$. 
\end{lemma} 

\begin{proof}

    For any $k \ge 1$, we have 
    \begin{align*}
        &\; \frac{\gamma_{k+1}}{2} \left\| \z_{k+1}' - \x^* \right\|^2 + f (\x_{k+1}) - f (\x^*) \\ 
        =& \;  
        \frac{\gamma_{k+1}}{2} \left\| \z_{k} - \frac{ \eta }{ \alpha_k} \g_k (\y_k) - \x^* \right\|^2 + f (\x_{k+1}) - f (\x^*) \\
        =& \; 
        \frac{\gamma_{k+1}}{2} \left\| \z_{k} - \x^* \right\|^2 - \< \frac{ \eta \gamma_{k+1} }{ \alpha_k} \g_k (\y_k) , \z_k - \x^* \> + \frac{ \gamma_{k+1} \eta^2 }{ 2 \alpha_k^2} \| \g_k (\y_k) \|^2 + f (\x_{k+1}) - f (\x^*) \\
        =& \; 
        \frac{\gamma_{k+1}}{2} \left\| \z_{k} - \x^* \right\|^2 - \< \alpha_k \g_k (\y_k) , \z_k - \x^* \> + \frac{ \eta }{2} \| \g_k (\y_k) \|^2 + f (\x_{k+1}) - f (\x^*)  \\
        \le& \; 
        \frac{\gamma_{k+1}}{2} \left\| \z_{k} - \x^* \right\|^2 - \< \alpha_k \g_k (\y_k) , \z_k - \x^* \> + f (\y_k) - f (\x^*) , 
    \end{align*}
    where the last inequality follows from Lemma \ref{lem:decrease-inexact}. 

    If $\x^* \in \mathbb{B}_\mu$, then $ \left\| \z_{k+1} - \x^* \right\| \le \left\| \z_{k+1}' - \x^* \right\| $. Thus 
    \begin{align*}
        \Delta_{k+1} 
        \le 
        \frac{\gamma_{k+1}}{2} \left\| \z_{k} - \x^* \right\|^2 - \< \alpha_k \g_k (\y_k) , \z_k - \x^* \> + f (\y_k) - f (\x^*) . 
    \end{align*}

    Let $ \E_k $ be the expectation with respect to $\xi_k$ and $\xi_k'$. 
    Since $ \mathcal{G} $ satisfies property \textbf{(H1)}, taking expectation (with respect to $\xi_k$ and $\xi_k'$) on both sides of the above inequality gives 
    \begin{align*} 
        \E_k \[ \Delta_{k+1} \] 
        \le& \;  
        \frac{\gamma_{k+1}}{2} \left\| \z_{k} - \x^* \right\|^2 - \alpha_k  \< {\nabla} f (\y_k) , \z_k - \x^* \> + C \alpha_k \epsilon_k \| \z_k - \x^* \|  +  f (\y_k)  - f (\x^*) , 
    \end{align*} 
    for some absolute constant $C$. Since $ \z_k \in \B_{\mu} $ and $\x^* \in \B_{\mu}$, we know $ \| \z_k - \x^* \| \le 2 \mu $, and thus 
    \begin{align} 
        \E_k \[ \Delta_{k+1} \] 
        \le& \;  
        \frac{\gamma_{k+1}}{2} \left\| \z_{k} - \x^* \right\|^2 + \< {\nabla} f (\y_k) , (1 - \alpha_k) \x_k + \alpha_k \x^* - \y_k \> + f (\y_k)   - f (\x^*) + 2 C \mu \alpha_k \epsilon_k \nonumber \\ 
        \le& \; 
        \frac{\gamma_{k+1}}{2} \left\| \z_{k} - \x^* \right\|^2 + (1 - \alpha_k) (f (\x_k ) - f (\x^*)) + 2 C \mu \alpha_k \epsilon_k , \nonumber \\
        =& \; 
        \frac{\gamma_{k+1}}{\gamma_{k}} \Delta_k + \( 1 - \alpha_k - \frac{\gamma_{k+1}}{\gamma_{k}}  \) ( f (\x_k) - f (\x^*) ) + 2 C \mu \alpha_k \epsilon_k . \label{eq:1}
    \end{align} 
    where the second last line follows from convexity of $f$.

    Then taking total expectation on both sides of (\ref{eq:1}) and rearranging terms gives 
    \begin{align*}
        \E \[ \( \frac{1}{\gamma_{k}} + \frac{\alpha_k}{\gamma_{k+1}} - \frac{1}{\gamma_{k+1}} \) \delta_k \] 
        \le 
        \E \[ \frac{\Delta_{k}}{\gamma_{k}} \] - \E \[ \frac{\Delta_{k+1}}{\gamma_{k+1}} \] + \frac{ 2 C \mu \alpha_k  \epsilon_k }{ \gamma_{k+1} }. 
    \end{align*} 
\end{proof}
\if\highlight1 
\color{red}
\fi 
With the above results in place, we are now ready to prove Theorem \ref{thm:main-inexact}. \color{black}

\begin{proof}[Proof of Theorem \ref{thm:main-inexact}] 

    Since $ p \ge 3 $, $\alpha_k = \frac{p}{k+p}$, and $ \gamma_{k+1} = \frac{\alpha_k^2}{\eta} $, we have that  
    \begin{align*} 
        \frac{1}{\gamma_{k}} + \frac{\alpha_k}{\gamma_{k+1}} - \frac{1}{ \gamma_{k+1} } 
        \ge \eta \frac{k+p}{  p^2 } . 
    \end{align*} 

     Since $ \sum_{k=0}^\infty \frac{\epsilon_k}{\alpha_k} < \infty $, Lemma \ref{lem:key} implies 
    \begin{align}
        \E \[ \sum_{k=1}^\infty (k+p) \delta_k \] 
        < \infty . \label{eq:key} 
    \end{align} 

    Since $ \E \[ \sum_{k=1}^\infty (k+p) \delta_k \] < \infty $, $ \sum_{k=1}^\infty (k+p) \delta_k < \infty $ almost surely. In addition, $ \delta_k $ is nonnegative and decreasing (by Lemma \ref{lem:decrease-inexact}). Therefore, by Proposition \ref{prop:basic}, we know that $ \lim \limits_{k \to \infty} k^2 ( f (\x_k) - f (\x^*) ) = 0$ almost surely. 

    By (\ref{eq:key}), we know that 
    \begin{align} 
        \Pr \( \sum_{k=1}^\infty (k+p) \delta_k < \infty \) = 1. \label{eq:contra-1}
    \end{align} 

    Now, fix $\kappa \in \mathbb{Q}_+$ and consider the event 
    \begin{align*} 
        \mathcal{E}_\kappa = \left\{ \delta_k \ge \frac{\kappa}{k^2 \log k \log \log k} , \; \text{ for all } k = 2,3,\cdots \right\} . 
    \end{align*} 
    
    Suppose, in order to get a contradiction, 
    that there exists $ \beta >0 $ 
    such that 
    \begin{align} 
        \Pr \( \mathcal{E}_\kappa \) > \beta . \label{eq:contra-2}
    \end{align} 

    When $\mathcal{E}_\kappa$ is true, we know that 
    \begin{align*} 
        \sum_{ k=2 }^\infty (k+p) \delta_k 
        \ge 
        \sum_{ k=2 }^\infty \frac{ \kappa }{ k \log k \log \log k } 
        = \infty . 
    \end{align*} 

    The above argument implies that, (\ref{eq:contra-2}) leads to a contradiction to (\ref{eq:contra-1}) for any $\beta > 0$. Therefore, $ \Pr \(\mathcal{E}_\kappa\) = 0 $ for any $\kappa \in \mathbb{Q}_+$. 
    Thus for any $\kappa \in \mathbb{Q}_+$, it holds that 
    \begin{align*} 
        \Pr \( \liminf_{ k \to \infty } k^2 \log k \log \log k (f (\x_k) - f ^*) 
        \ge \kappa \) = 0 . 
    \end{align*} 
    Thus we have 
    \begin{align*} 
        &\; \Pr \( \liminf_{ k \to \infty } k^2 \log k \log \log k (f (\x_k) - f ^*) 
        > 0 \) \\ 
        =& \; 
        \Pr \( \cup_{\kappa \in \mathbb{Q}_+} \left\{ \liminf_{ k \to \infty } k^2 \log k \log \log k (f (\x_k) - f ^*) 
        \ge \kappa \right\} \) \\ 
        \le& \; 
        \sum_{\kappa \in \mathbb{Q}_+} \Pr \(  \liminf_{ k \to \infty } k^2 \log k \log \log k (f (\x_k) - f ^*) 
        \ge \kappa \) 
        = 0 . 
    \end{align*} 
    This concludes the proof of Theorem \ref{thm:main-inexact}. 
    

\end{proof}  

\section{NSA for Nonsmooth Objective} 
\label{sec:nonsmooth}

Another important problem in optimization is when the objective is convex but possibly nonsmooth. Various methods have been invented to handle the possible nonsmoothness in the objective, including subgradient methods and proximal methods. In this section, we study a proximal version of the NSA algorithm.

In a Hilbert space $\mathcal{H}$ with inner product $ \< \cdot, \cdot \>$ and norm $\| \cdot \|$, with respect to proper lower-semicontinuous convex function  $ \Psi : \mathcal{H} \to \R $ and step size $ \eta $, the Moreau--Yosida proximal operator is defined as a map from $ \mathcal{H} $ to $\mathcal{H}$ such that 
\begin{align} 
    \prox_{\eta, \Psi} (\x) = \arg\min_{\y \in \mathcal{H}} \left\{ \Psi (\y) + \frac{1}{2\eta} \| \y - \x \|^2 \right\} . 
\end{align} 


For a convex proper lower-semicontinuous objective function $f$. This proximal version of the NSA algorithm is summarized below: 
\begin{align} 
    \begin{split}
        \x_{k+1} =& \;  \prox_{s, f } \( (1 - \alpha_k) \x_k + \alpha_k \z_k \) \\ 
        \z_{k+1} =& \; \z_k + \frac{\eta}{s\alpha_k} \( \x_{k+1} - \big( (1 - \alpha_k) \x_k + \alpha_k \z_k \big) \) , 
    \end{split} \label{eq:nsa-prox}
\end{align}
where $ \alpha_k = \frac{p}{k+p} $ for some fixed constant $p$, $s, \eta$ are step sizes, and $ \x_0 = \z_0 $ are initialized to be the same point. The convergence rate of (\ref{eq:nsa-prox}) is in Theorem \ref{thm:rate-prox}.

\begin{theorem}
    \label{thm:rate-prox} 
    Let the objective be a proper lower semicontinuous convex function. Let $ \alpha_k = \frac{p}{k+p} $ with $p \ge 3$. If $s \ge \frac{\eta}{2}$, then there exists a constant $c$ independent of $k$ (depending on the initialization and step sizes), such that the proximal version of the NSA algorithm (\ref{eq:nsa-prox}) satisfies $ f (\x_k) - f^* \le \frac{c}{k^2} $ for all $k \in \mathbb{N}$, where $f^* > -\infty$ is the minimum value of $f$. 
\end{theorem}

\begin{proof}
    Let $\partial f (\x)$ be the set of subgradients of $f$ at $\x$. 
    By the optimality condition for subgradients (e.g., \cite{RockWets98}), there exists $ \g_{k+1} \in \partial f (\x_{k+1}) $, such that 
    \begin{align*}
        \x_{k+1} =  (1 - \alpha_k) \x_k + \alpha_k \z_k - s \g_{k+1} . 
    \end{align*}

    Introduce $\gamma_{k+1} = \frac{\alpha_k^2}{\eta} $ for $k = 1,2,\cdots$. Let $\x^*$ be a point where $f^*$ is attained. Define 
    \begin{align*}
        \Delta_k := \frac{\gamma_{k}}{2} \| \z_k - \x^* \|^2 + f (\x_k ) - f (\x^*) . 
    \end{align*}

    Then we have 
    \begin{align*}
        \Delta_{k+1} 
        =& \;  
        \frac{\gamma_{k+1}}{2} \left\| \z_{k+1} - \x^* \right\|^2 + f (\x_{k+1} ) - f (\x^*) \\ 
        =& \; 
        \frac{\gamma_{k+1}}{2} \left\| \z_k - \x^* \right\|^2 + \< \g_{k+1} , \alpha_k \x^* - \alpha_k \z_k \> + \frac{\eta}{2} \| \g_{k+1} \|^2 + f (\x_{k+1} ) - f (\x^*) \\ 
        =& \; 
        \frac{\gamma_{k+1}}{2} \left\| \z_k - \x^* \right\|^2 + \< \g_{k+1} , \alpha_k \x^* + (1 - \alpha_k ) \x_k - s \g_{k+1}-\x_{k+1} \> \\
        &+ \frac{\eta}{2} \| \g_{k+1} \|^2 + f (\x_{k+1} ) - f (\x^*) \\ 
        \le&\; 
        \frac{\gamma_{k+1}}{2} \left\| \z_k - \x^* \right\|^2 + \< \g_{k+1} , \alpha_k \x^* + (1 - \alpha_k ) \x_k -\x_{k+1} \> + f (\x_{k+1} ) - f (\x^*) \tag{since $s \ge \frac{\eta}{2}$} \\
        \le&\; 
        \frac{\gamma_{k+1}}{2} \left\| \z_k - \x^* \right\|^2 + f (\alpha_k \x^* + (1 - \alpha_k) \x_k) - f (\x^*) \\
        \le&\; 
        \frac{\gamma_{k+1}}{2} \left\| \z_k - \x^* \right\|^2 + (1 - \alpha_k) (f (\x_k) - f (\x^*)), 
    \end{align*} 
    where the last two lines use convexity of the function $f$. Rearranging terms gives that $ \frac{\Delta_{k+1}}{\gamma_{k+1}} \le \frac{\Delta_{k}}{\gamma_{k}}  $. Therefore, 
    \begin{align*}
        \frac{\Delta_{k+1}}{\gamma_{k+1}} \le \frac{\Delta_{k}}{\gamma_{k}} \cdots \le \frac{\Delta_{1}}{\gamma_{1}}, 
    \end{align*}
    which concludes the proof since $ \gamma_{k} = \Theta \( 1/k^2 \) $. 
\end{proof}


\subsection{A Continuous-time Perspective}
\label{sec:cont}


In this section, we provide a continuous-time analysis for a proximal version of the NSA algorithm (\ref{eq:nsa-prox}). 
In order for us to find a proper continuous-time trajectory of the discrete-time algorithm, Taylor approximation in Hilbert spaces will be used. Therefore, we want the objective function to be smooth. Otherwise, the error terms can get out-of-control and the correspondence between the discrete-time algorithm and continuous-time trajectory may break. After an ODE is derived for the smooth case via Taylor approximation, we could analogously write out an differential inclusion for the nonsmooth case, with derivatives defined in a subdifferential sense (e.g., \cite{doi:10.1137/15M1046095}).  

When the objective function $f$ is differentiable, we rewrite (\ref{eq:nsa-prox}) as 
\begin{align} 
    \begin{split}
        \x_{k+1} &= (1 - \alpha_k) \x_k + \alpha_k \z_k - s \nabla f (\x_{k+1}) \\
        \z_{k+1} &= \z_k - \frac{\eta}{\alpha_k} \nabla f (\x_{k+1}) , 
    \end{split} \label{eq:nsa-prox-2} 
\end{align} 
where $\alpha_k = \frac{p}{k+p}$, and $s,\eta$ are step sizes. 

Let $\{ \x_k \}_{k\in\mathbb{N}}$ and $\{ \z_k \}_{k\in\mathbb{N}}$ be associated with two curves $X(t)$ and $Z(t)$ ($t\ge 0$). 
For the first equation in (\ref{eq:nsa-prox-2}), we introduce a correspondence between discrete-time index $k$ and continuous time index $t$: $t = k \sqrt{\eta}$. With this we have $X(t) \approx \x_k$, $ X(t+\sqrt{\eta}) \approx \x_{k+1} $, and $\alpha_k = \frac{p}{k+p} = \frac{\sqrt{\eta} p}{t + \sqrt{\eta} p} = \frac{ \sqrt{\eta} p}{ t} + o(\sqrt{\eta}) $. \if\highlight1 \color{red} \fi Henceforth, $X$ and $X(t)$ (respectively, $\dot{X}$ and $\dot{X}(t)$) are used interchangeably where the meaning is clear. \color{black} It holds that 
\begin{align*} 
    \x_{k+1} - \x_k = \sqrt{\eta} \dot{X}(t) + o(\sqrt{\eta}),  
\end{align*} 
and when $s = o (1)$, we Taylor approximate the trajectories in the Hilbert space $\mathcal{H}$ \cite{cartan2017differential} 
\begin{align*}
    &\; \alpha_k (\x_k - \z_k) + s \nabla f (\x_{k+1}) \\ 
    =& \;  
    \frac{ \sqrt{\eta} p }{ t + \sqrt{\eta} p} \( X(t) - Z (t) \) + s \nabla f (X(t) + \sqrt{ \eta } \dot{X} + o(\sqrt{\eta})) \\ 
    =& \;  
    \frac{ \sqrt{\eta} p }{ t } \( X(t) - Z (t) \) + \frac{s}{\sqrt{\eta}} \sqrt{\eta} \nabla f (X(t) ) + s \sqrt{\eta} \nabla^2 f ( X ) \dot{X} + o( \sqrt{\eta} ) \\ 
    =& \; 
    \frac{ \sqrt{\eta} p }{ t } \( X(t) - Z (t) \) + \frac{s}{\sqrt{\eta}} \sqrt{\eta} \nabla f (X(t) ) + o( \sqrt{\eta} ) , 
\end{align*}
\if\highlight1 
\color{red}
\fi 
where the last line uses that $ O ( s \sqrt{\eta} ) = o (\sqrt{\eta}) $. \color{black}

Collecting terms multiplied to $\sqrt{\eta}$, and we obtain that when $\eta$ and $s$ are small, the first equation in (\ref{eq:nsa-prox-2}) corresponds to the following ODE: \if\highlight1 
\color{red}
\fi 
\begin{align} 
    \dot{X} + \frac{p}{t} (X - Z) + \frac{s}{\sqrt{\eta}} \nabla f (X) = 0. \label{eq:ode-1} 
\end{align} 
\color{black}
For the second equation in (\ref{eq:nsa-prox-2}), we again introduce a correspondence between discrete-time index $k$ and continuous time index $t$: $t = k \sqrt{\eta}$. With this we have $Z (t+\sqrt{\eta}) \approx \z_{k+1}$ and $\frac{1}{\alpha_k} = \frac{t + \sqrt{\eta} p}{\sqrt{\eta} p} $. Thus we have 
\begin{align*} 
    \z_{k+1} - \z_k 
    = 
    Z (t + \sqrt{\eta}) - Z(t) 
    = 
    \sqrt{\eta } \dot{Z} + o (\eta) , 
\end{align*} 
and 
\begin{align*}
    \frac{\eta}{\alpha_k } \nabla f (\x_{k+1}) 
    = 
    \sqrt{\eta} \frac{t}{p} \nabla f (X (t)) + o (\sqrt{\eta}). 
\end{align*}

Collecting terms, and we obtain that when $\eta$ is small, the second equation in (\ref{eq:nsa-prox-2}) corresponds to the following ODE: 
\begin{align*} 
    \dot{Z} + \frac{t}{p} \nabla f (X) = 0. 
\end{align*}

Thus the rule (\ref{eq:nsa-prox-2}) corresponds to the following system of ODEs: 
\begin{align} 
    \begin{split}
        &t \dot{X} + p (X - Z) + t \frac{s}{\sqrt{\eta}} \nabla f (X) = 0, \\  
    &\dot{Z} + \frac{t}{p} \nabla f (X) = 0 . 
    \end{split} \label{eq:sys-ode}
\end{align} 


In this system of ODEs, the trajectory of $X$ converges to the optimum at rate $ O (1/t^2) $ for convex functions. 

\begin{proposition} 
    \label{prop:cont}
    Let $X$, $Z$ be governed by the system of ODEs (\ref{eq:sys-ode}) with a differentiable convex function $f \in \mathscr{F}$. If $p \ge 2$, then there exists a time-independent constant $c$ (depending only on $X(0)$,$f (X(0))$ and $s,\eta$) such that $ f (X (t)) - f^* \le \frac{c}{t^2} $. 
\end{proposition} 


    
\begin{proof}[Proof of Proposition \ref{prop:cont}]
    
    Recall $\alpha_k = \frac{p}{k+p}$ in (\ref{eq:sys-ode}). Define the energy functional 
    \begin{align*}
        \mathcal{E} (t) = \frac{p^2}{2} \| Z - \x^* \|^2 + f (X) - f (\x^*) . 
    \end{align*} 
    
    For this energy functional, we have 
    \begin{align*}
        \frac{d \mathcal{E} (t) }{dt} 
        =& \;  
        p^2 \< Z - \x^*, \dot{Z} \> + 2 t \( f (X) - f (\x^*) \) + t^2 \< \nabla f (X) , \dot{X} \> \\ 
        \le& \; 
        p^2 \< Z - \x^*, \dot{Z} \> + 2 t \< \nabla f (X) , X - \x^* \> + t^2 \< \nabla f (X) , \dot{X} \> \tag{by convexity} \\ 
        =& \; 
        p \< Z - \x^*, - t \nabla f (X) \> + 2 t \< \nabla f (X) , X - \x^* \> \\ 
        &+ t \< \nabla f (X) , p (Z - X) - \frac{t s}{\sqrt{\eta}}\nabla f (X) \>  \\ 
        =& \; 
        - \frac{t^2 s }{ \sqrt{\eta} } \| \nabla f (X) \|^2 + \( p - 2 \) t \< \nabla f (X), \x^* - X \> 
        \\
        \le& \; 0. 
    \end{align*} 
    The above implies 
    \begin{equation}
       f(X(t))-f(\x^*)\leq \frac{\mathcal{E}(t)}{t^2}\leq\frac{\mathcal{E}(0)}{t^2}, \nonumber
    \end{equation}
    which concludes the proof. 
\end{proof}

\subsubsection{A Single Second-order ODE for (\ref{eq:nsa-prox-2})} 
\label{sec:ode}

We can combine the above two ODEs (\ref{eq:sys-ode}) and get a second-order ODE governing the trajectory of $X$. \if\highlight1 
\color{red}
\fi Specifically, we take derivatives on both sides of the first equation in (\ref{eq:sys-ode}), and substitute the $\dot{Z}$ term by the second equation in (\ref{eq:sys-ode}). This derivation gives the following second-order ODE: \color{black} 
\begin{align} 
    \ddot{X} +  \( \frac{s}{\sqrt{\eta}} \nabla^2 f (X) + \frac{p+1}{t} I \) \dot{X} + \( 1 + \frac{s}{t \sqrt{\eta }} \) \nabla f (X) = 0 , \label{eq:high-resolution}
\end{align} 
where $I$ denotes the identity element in $ \mathcal{H} \otimes \mathcal{H}^* $. Recall $\mathcal{H}$ is the space over which $f$ is defined. 
 
Note that the single second-order ODE recovers the high-resolution ODE derived in \cite{2018Understanding}. Continuous-time analysis of gradient and/or proximal methods themselves form a rich body of the literature. \cite{doi:10.1080/01630560008816971,alvarez2001inertial,ALVAREZ2002747,attouch2018fast} were among the first that formally connects physical systems to gradient algorithms. Principles of classical mechanics \cite{arnold1989mathematical} have also been applied to this topic \cite{ATTOUCH20165734,doi:10.1073/pnas.1614734113,JMLR:v22:20-195}. Also, \cite{NEURIPS2018_44968aec} shows that refined discretization in the theorem of numeric ODEs scheme achieves acceleration for smooth enough functions. 
Till now, the finest correspondence between continuous-time ODE and discrete-time Nesterov acceleration is find in \cite{2018Understanding}. 
In this paper, we show that NSA easily recovers this high-resolution ODE. 


\if\highlight1 
\color{red}
\fi 


\color{black}

\if\highlight1 
\color{red}
\fi 
\section{NSA for Composite Objectives}
\label{sec:composite}
    
In many scenarios, we encounter composite objective problems.
Such problems are of the following form
\begin{align}
    \min_{\x} F (\x) , \quad F (\x) = f (\x) + h (\x), \label{eq:obj-composite}
\end{align}
where $f$ is smooth convex and $h$ is nonsmooth and convex. This class of problems includes those where $ h$ serves as a constraint to a convex set or where $ h $ represents the $ L_1 $-norm. In such cases, we can slightly alter the NSA algorithm so that it handles composite objectives. This version of NSA is summarized below in Algorithm \ref{alg:nsa-comp}. 

\begin{algorithm}[H] 
    \caption{Nesterov--Spokoiny Acceleration (NSA) for Composite Objectives} 
    \label{alg:nsa-comp} 
    \begin{algorithmic}[1]  
        \STATE \textbf{Initialize: }Pick $ \x_0 = \z_0 \in \R^n $ (such that $ \partial F (\x_0) \not\ni 0 $.) 
        \FOR{$k = 0,1,2,\cdots,$}
            \STATE Compute $\y_{k} = (1 - \alpha_k) \x_k + \alpha_k \z_k $, with $\alpha_k = \frac{p}{k+p}$ for some constant $p$. 
            \STATE $\x_{k+1}' = \prox_{\eta, h} \( \y_k - \eta \nabla f (\y_k) \) $. /* $\eta$ is the step size. */
            \STATE $\x_{k+1}'' = \prox_{\eta, h} \( \x_k - \eta \nabla f (\x_k) \) $. 
            \STATE 
                $
                    \x_{k+1} =
                    \begin{cases}
                        \x_{k+1}', & \text{ if } F ( \x_{k+1}' ) \le F ( \x_{k+1}'' ) \\ 
                        \x_{k+1}'', & \text{ otherwise. } 
                    \end{cases} 
                $ 
            \STATE $\z_{k+1} = \z_{k} - \frac{\eta}{\alpha_k} \nabla f (\y_k) - \frac{\eta}{\alpha_k} \g_{k+1}'$, where $ \eta \g_{k+1}' = \y_k - \eta \nabla f (\y_k) - \x_{k+1}' $
        \ENDFOR 
    \end{algorithmic} 
\end{algorithm}  

By the properties of the proximal operator, we know there exists vectors $\g_{k+1}' \in \partial h (\x_{k+1}') $ and $\g_{k+1}'' \in \partial h (\x_{k+1}'') $ such that 
\begin{align*} 
    \x_{k+1}' = \y_k - \eta \nabla f (\y_k) - \eta \g_{k+1}' 
    \quad \text{and} \quad
    \x_{k+1}'' = \x_k - \eta \nabla f (\x_k) - \eta \g_{k+1}'' . 
\end{align*} 

Algorithm \ref{alg:nsa-comp} also satisfies the descent property, as stated below in Lemma \ref{lem:compoosite-descent}. 

\begin{lemma}
    \label{lem:compoosite-descent}
    Consider the composite program in (\ref{eq:obj-composite}), where $f$ is convex, $L$-smooth and $h$ is convex and nonsmooth, and $F$ admits a minimum at $\x^*$ such that $F (\x^*) > -\infty$. 
    Let $\delta_k = F (\x_k) - F^* $. 
    Then Algorithm \ref{alg:nsa-comp} with $\eta = \frac{1}{L}$ satisfies that 
    \begin{align} 
        \delta_{k+1} \le \delta_k \label{eq:composite-descent}
    \end{align}
    and 
    \begin{align}
        F (\x_{k+1}') \le f (\y_k ) - \eta \< \nabla f (\y_k) , \nabla f (\y_k) + \g_{k+1}' \> + \frac{\eta}{2} \| \nabla f (\y_k) + \g_{k+1}' \|^2 + h (\x_{k+1}' ) . \label{eq:composiite-descent-2}
    \end{align} 
\end{lemma} 

\begin{proof}   
    By the algorithm update rule, we know $\x_{k+1}''$ is a prox-gradient step starting from $\x_k$. By the descent property (e.g., Chapter 10 in \cite{beck2017first}), we know 
    \begin{align*}
        F (\x_{k+1}) \le F (\x_{k+1}'') \le F ( \x_k  ) , 
    \end{align*}
    which proves (\ref{eq:composite-descent}). To prove (\ref{eq:composiite-descent-2}), we notice that 
    \begin{align*} 
        F (\x_{k+1}') =&\; f (\x_{k+1}') + h (\x_{k+1}') 
        =
        f (\y_{k} - \eta \nabla f (\y_k) - \eta \g_{k+1}' ) + h (\x_{k+1}') \\ 
        \le& \;  
        f (\y_k) - \eta \< \nabla f (\y_k) , \nabla f (\y_k) + \g_{k+1}' \> + \frac{ \eta }{2} \| \nabla f (\y_k) + \g_{k+1}' \|^2 + h (\x_{k+1}) , 
    \end{align*} 
    where the last inequality uses $L$-smoothness of $f$. Setting $\eta = \frac{1}{L}$ finishes the proof. 
\end{proof} 




    



We are now ready to present and prove the convergence guarantee for Algorithm \ref{alg:nsa-comp}. 

\begin{theorem}
    \label{thm:prox-improve}
    Consider the composite program in (\ref{eq:obj-composite}), where $f$ is convex, $L$-smooth and $h$ is convex and nonsmooth, and $F$ admits a minimum at $\x^*$ such that $F^* = F (\x^*) > -\infty$. 
    Let $\delta_k = F (\x_k) - F^* $. 
    Then Algorithm \ref{alg:nsa-comp} with $\eta = \frac{1}{L}$ satisfies that 
    \begin{align*} 
        \lim_{k\to \infty } k^2 \delta_{k} = 0. 
    \end{align*}
\end{theorem}

\begin{proof}
    Define  
    \begin{align*} 
        \Delta_k := \frac{\gamma_k}{2} \| \z_k - \x^* \|^2 + F (\x_k) - F (\x^*) , 
    \end{align*} 
    where $\gamma_{k+1} := \frac{\alpha_k^2}{\eta}$. 

    Then by the algorithm rule, we have 
    \begin{align}
        \Delta_{k+1} 
        =& \;  
        \frac{\gamma_{k+1}}{2} \| \z_{k+1} - \x^* \|^2 + f (\x_{k+1}) + h (\x_{k+1}) - f (\x^*) - h (\x^*) \nonumber \\ 
        \le& \; 
        \frac{\gamma_{k+1}}{2} \left\| \z_{k} - \x^* \right\|^2 - \frac{\eta \gamma_{k+1}}{\alpha_k} \< \nabla f (\y_k) + \g_{k+1}' , \z_k - \x^* \> \nonumber \\ 
        & + \frac{\eta^2 \gamma_{k+1}}{2\alpha_k^2} \| \nabla f (\y_k) + \g_{k+1}' \|^2 + F (\x_{k+1}') - F (\x^*) \nonumber \\ 
        =& \; 
        \frac{\gamma_{k+1}}{2} \left\| \z_{k} - \x^* \right\|^2 +  \< \nabla f (\y_k) + \g_{k+1}' , \alpha_k \x^* + (1 - \alpha_k ) \x_k - \y_k \> \nonumber \\ 
        & + \frac{\eta }{2 } \| \nabla f (\y_t) + \g_{k+1}' \|^2 + F (\x_{k+1}') - F (\x^*) . \label{eq:composite-1}
    \end{align} 

    By the descent property in Lemma \ref{lem:compoosite-descent}, we have 
    \begin{align}
        F (\x_{k+1}') 
        \le& \;  
        f (\y_k) + \< \nabla f (\y_k) , - \eta \nabla f (\y_k) - \eta \g_{k+1}' \> + \frac{ \eta }{2} \| \nabla f (\y_k) + \g_{k+1}' \|^2 + h (\x_{k+1}') \nonumber \\ 
        =& \; 
        f (\y_k) + h (\x_{k+1}') - \frac{\eta}{2} \| \nabla f (\y_k) + \g_{k+1}' \|^2 + \eta \< \nabla f (\y_k) + \g_{k+1}', \g_{k+1}' \> \label{eq:composite-2}
    \end{align} 

    Plugging (\ref{eq:composite-2}) into (\ref{eq:composite-1}) gives 
    \begin{align*}
        \Delta_{k+1} 
        \le& \; 
        \frac{\gamma_{k+1}}{2} \left\| \z_{k} - \x^* \right\|^2 +  \< \nabla f (\y_k) + \g_{k+1}' , \alpha_k \x^* + (1 - \alpha_k ) \x_k - \y_k \> \\ 
        & + \eta \< \nabla f (\y_k) + \g_{k+1}', \g_{k+1}' \> + f (\y_{k}) + h ( \x_{k+1}' ) - F (\x^*) \\ 
        {\le}& \; 
        \frac{\gamma_{k+1}}{2} \left\| \z_{k} - \x^* \right\|^2 + f ( \alpha_k \x^* + (1 - \alpha_k ) \x_k ) + h ( \x_{k+1}' ) + \eta \< \nabla f (\y_k) + \g_{k+1}', \g_{k+1}' \> \\ 
        & + \<  \g_{k+1}' , \alpha_k \x^* + (1 - \alpha_k ) \x_k - \x_{k+1}' - \eta \nabla f (\y_k) - \eta \g_{k+1}' \> - F (\x^*) , 
    \end{align*}
    where the last inequality uses convexity of $f$. Simplifying terms in the above inequality gives 
    \begin{align*}
        \Delta_{k+1} 
        \le& \;
        \frac{\gamma_{k+1}}{2} \left\| \z_{k} - \x^* \right\|^2  + f ( \alpha_k \x^* + (1 - \alpha_k ) \x_k ) + h ( \x_{k+1}' )  \\ 
        & + \<  \g_{k+1}' , \alpha_k \x^* + (1 - \alpha_k ) \x_k - \x_{k+1}' \> - F (\x^*) \\
        \le& \; 
        \frac{\gamma_{k+1}}{2} \left\| \z_{k} - \x^* \right\|^2 + f ( \alpha_k \x^* + (1 - \alpha_k ) \x_k ) + h ( \alpha_k \x^* + (1 - \alpha_k ) \x_k ) - F (\x^*) , 
    \end{align*}
    where the last inequality uses convexity of $h$. Another use of convex properties of $f$ and $h$ gives 
    \begin{align*}
        \Delta_{k+1} 
        \le 
        \frac{\gamma_{k+1}}{2} \left\| \z_{k} - \x^* \right\|^2 + \( 1 - \alpha_k \) \( F (\x_k) - F (\x^*) \) . 
    \end{align*}

    Rearranging terms in the above inequality gives 
    \begin{align} 
        \frac{\Delta_{k+1} }{\gamma_{k+1}}
        \le& \;  
        \frac{\Delta_k }{\gamma_k} + \( \frac{1}{\gamma_{k+1}} - \frac{\alpha_k}{\gamma_{k+1}} - \frac{ 1 }{\gamma_k}  \) ( F (\x_k) - F (\x^*) ) . 
        \label{eq:0-composite} 
    \end{align} 
    
    Since $\alpha_k = \frac{p}{k+p}$ for some $p \ge 3$, we have that  
    \begin{align*} 
        \frac{1}{\gamma_{k+1}} - \frac{\alpha_k}{\gamma_{k+1}} - \frac{ 1 }{\gamma_k} 
        =& \;  
        \frac{ \eta (k+p)^2 }{ p^2} - \frac{\eta (k+p)}{p} - \frac{ \eta (k+p-1)^2 }{ p^2 } \\ 
        =& \;  
        \eta \frac{ \( 2 - p \) ( k+p ) - 1 }{  p^2 } \le - \eta \frac{k+p}{ p^2} . 
    \end{align*} 

    Recall we write $\delta_k := F (\x_k) - F (\x^*)$ for simplicity. Rearranging terms in (\ref{eq:0-composite}) gives 
    \begin{align*}
        \eta \frac{k+p}{ p^2} \delta_k \le \frac{\Delta_k}{\gamma_k} - \frac{\Delta_{k+1}}{\gamma_{k+1}}, 
    \end{align*}
    and thus 
    \begin{align} 
        \sum_{k=1}^\infty \frac{k+p}{ p^2} \delta_k < \infty . \label{eq:sum-composite}
    \end{align} 
    By Lemma \ref{lem:decrease}, we know that $ \{ \delta_k \}_{k=1}^\infty $ is decreasing. 
    Thus by Proposition \ref{prop:basic}, we know $  \limsup\limits_{k \to \infty} k^2 \delta_k = \lim\limits_{k \to \infty} k^2 \delta_k = 0 $. 
    
\end{proof}

\color{black}




\section{Experiments}
We compare the NSA algorithm (Algorithm \ref{alg:nsa-comp}) with classic gradient descent and benchmark acceleration methods, including the AFBM method \cite{nesterov1983method,Nesterov2013,doi:10.1137/15M1046095}, the FISTA algorithm \cite{2009A}, and the NSA algorithm in its original form \cite{2017Random}. It has been empirically observed that NAG with over-damping may converge faster \citep{JMLR:v17:15-084}. For this reason, we compare NSA with possibly over-damped acceleration methods. 
Typically, the damping factor $\ge 3$ yields better empirical performance. 
In all experiments, we pick the same values of learning rate $\eta$ for all methods.

The start point $\x_0$ is also same for all algorithms. In the first series of experiments (Figure \ref{fig:two}), we assess the methods across various standard machine learning and optimization tasks, such as regression, logistic regression, and matrix completion. In the second series of experiments (Figure \ref{fig:three}), we demonstrate the benefits of these methods by applying them to a standard fully-connected neural network trained on the Iris dataset \cite{iris_53}, highlighting their utility in key real-world applications.



More experiments for the zeroth-order version of NSA (Algorithm \ref{alg:nsa-inexact} with zeroth-order gradient estimators) can be found in the Appendix. 


\begin{enumerate}[label=\textbf{Figure \ref{fig:two}\alph*:}, leftmargin=2\parindent, align=left] 
    \item The cost function of least square regression
    \begin{equation*}
        f(\x) = \frac{1}{2}\|A \x- \b \|^2,
    \end{equation*}
    where $A$ is a full-rank $400 \times 200$ random matrix   with entries sampled as $i.i.d.$ standard Gaussian, $ \: \mathcal{N}(0,1)$ and $\b$ is a $400 \times 1$ standard Gaussian vector. We set $\eta=0.0005$ for all algorithms.
    \item The cost function of logistics regression
    \begin{equation*}
    f(\x)=\sum_{i=1}^{200}\left[- y_i A_i^\top \x+ \log(1+e^{A_i^\top \x})\right],
    \end{equation*} 
    where $A =\left(A_1 , ..., A_{200}\right)^\top$ is a $200\times 5$ random matrix sampled from $i.i.d.\:  \mathcal{N}(0,1)$ and $\y = (y_1, \cdots, y_{200})^\top$ is a $200\times 1$ random vector of labels sampled from $i.i.d.\: Ber(1,0.5)$. We set $\eta=0.005$ for all algorithms.
    \item The cost function of  lasso regression
    \begin{equation*}
        f(\x) = \frac{1}{2} \|A \x-\b\|^2+\lambda \|\x\|_{1},
    \end{equation*}
     where $A$ is a full-rank $400 \times 200$ random matrix   with entries sampled as $i.i.d.$ standard Gaussian, $ \: \mathcal{N}(0,1)$ and $\b$ is a $400 \times 1$ standard Gaussian vector. We set $\eta=0.0005$ for all algorithms. 
    \item  The \texttt{log\_exp\_sum} function
    \begin{equation*}
    f(\x) = \rho \log \left[ \sum_{i=1}^{n}e^{(A_i^\top \x-b_i)/\rho}\right],
    \end{equation*}
    where $A = \left[ A_1,...A_{200}\right]^T$ is a $40\times 10$ random matrix sampled from $i.i.d. \:\mathcal{N}(0,1)$ and $\b = (b_1, b_2, \cdots, b_{200})^\top$ is a $40\times 1$ random vector  sampled from $i.i.d.\: \mathcal{N}(0,1)$. Let $ \rho = 5$. We set $\eta=0.5$ for all algorithms.
    \item  The cost function of  ridge regression
    \begin{equation*}
        f(\x) = \frac{1}{2} \|A\x-\b\|^2+\lambda \|\x\|_{2}^2,
    \end{equation*}
     where $A$ is a full-rank $400 \times 200$ random matrix   with entries sampled as $i.i.d.$ standard Gaussian, $ \: \mathcal{N}(0,1)$ and $\b$ is a $400 \times 1$ standard Gaussian vector. We set $\eta=0.0005$ for all algorithms. 
    \item  The cost function of matrix completion
    \begin{equation*}
    f(X)=\frac{1}{2}\|X_{ob}-A_{ob}\|^2+\lambda\|X\|_{*},
    \end{equation*}
    where  $A_{ob}$ includes the available elements in matrix $A$ and only 20\% of the elements in $A$ can be observed. The position of these elements is randomly chosen and the original matrix $A$ is a rank-3  $50\times 40$ matrix with the eigenvalues \textemdash$\left[1,2,3\right]$. To generate $A$, we  generate $U,S$ and $V$ where $S=diag\{1,2,3\}$ and $U,V^\top$ are rank-3 $50\times7$ matrices generated from $i.i.d \: Uniform(0,1)$. Then we calculate $A$ by $A = USV$, and set $\lambda=0.05$. We set $\eta=1$ for all algorithms.
\end{enumerate} 

\begin{figure}[H]
     \centering
     \subfloat[][$\min_{\x} \frac{1}{2}\|A \x - \b\|^2$]{\includegraphics[width = 0.45\linewidth]{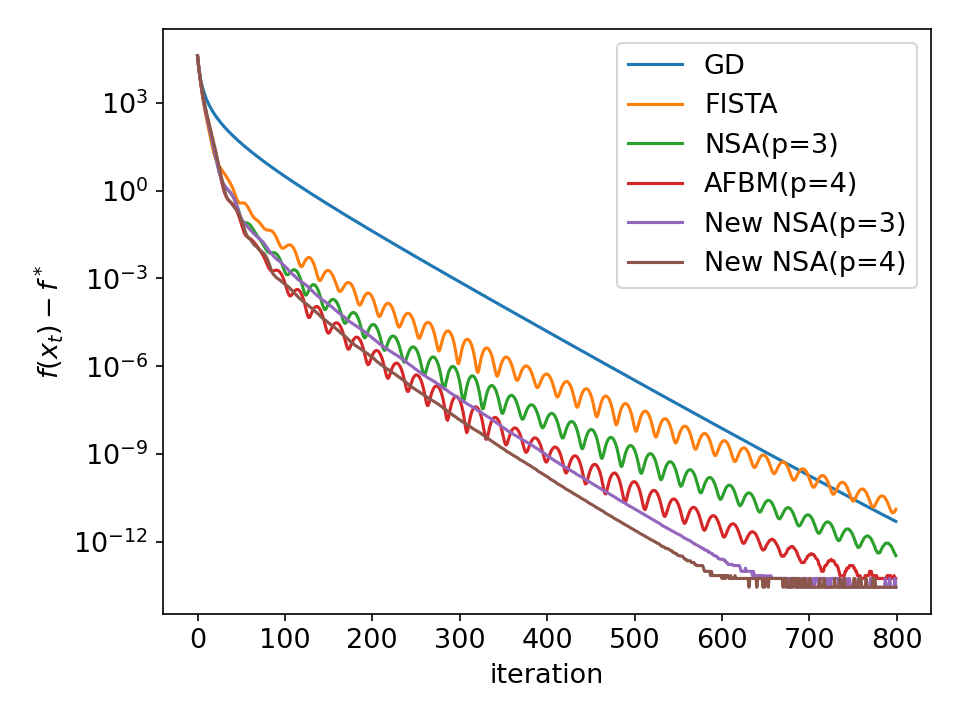} \label{fig1}} \hfill
     \subfloat[][$ \min_{\x} \sum_{i=1}^{n}-y_iA_i^\top \x+ \log(1+e^{A_i^\top \x}) $]{\includegraphics[width = 0.45\linewidth]{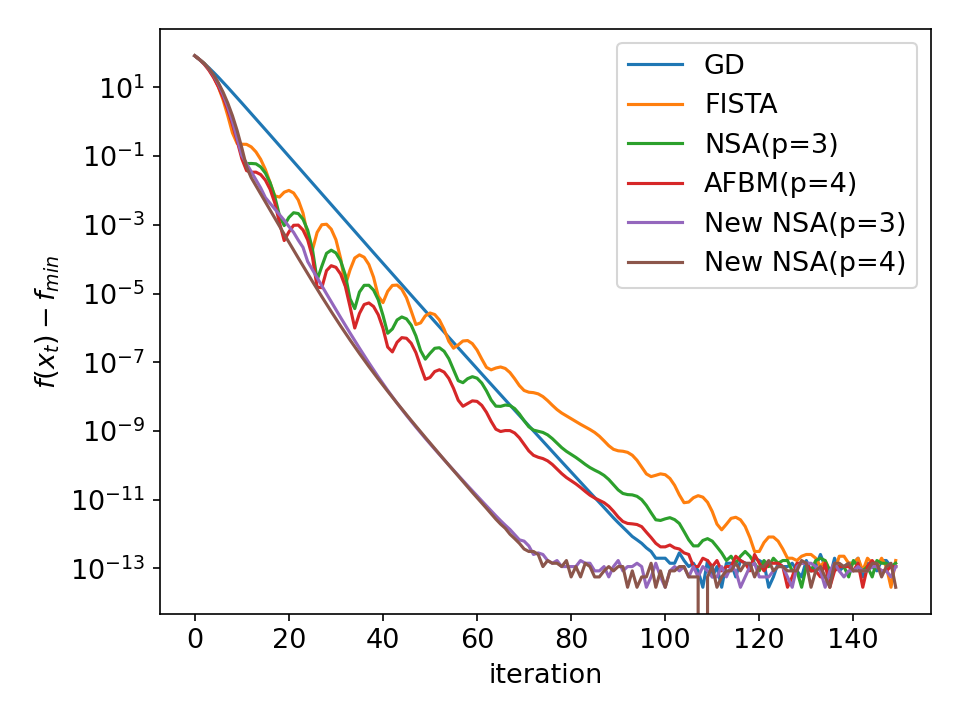} \label{fig2}} \\ 
     \subfloat[][$ \min_{\x} \frac{1}{2} \|A \x- \b\|^2+\lambda \|x\|_{1} $]{\includegraphics[width = 0.45\linewidth]{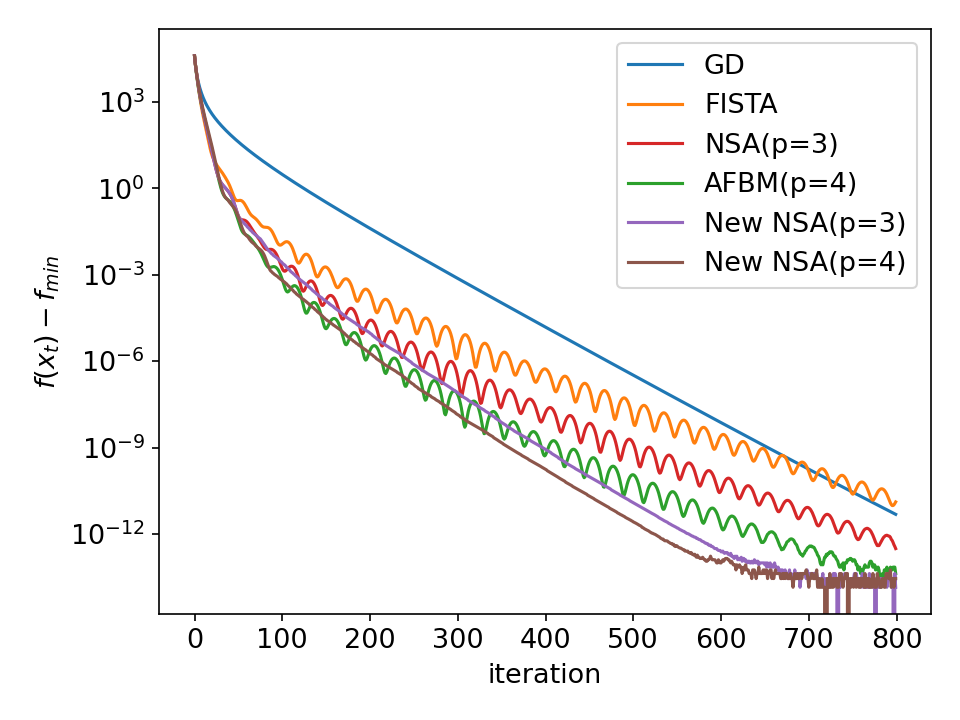} \label{fig3} } \hfill
     \subfloat[][$ \min_{\x} \rho \log \[ \sum_{i=1}^{n} \exp(A_i^\top \x-b_i)/\rho \] $]{\includegraphics[width = 0.45\linewidth]{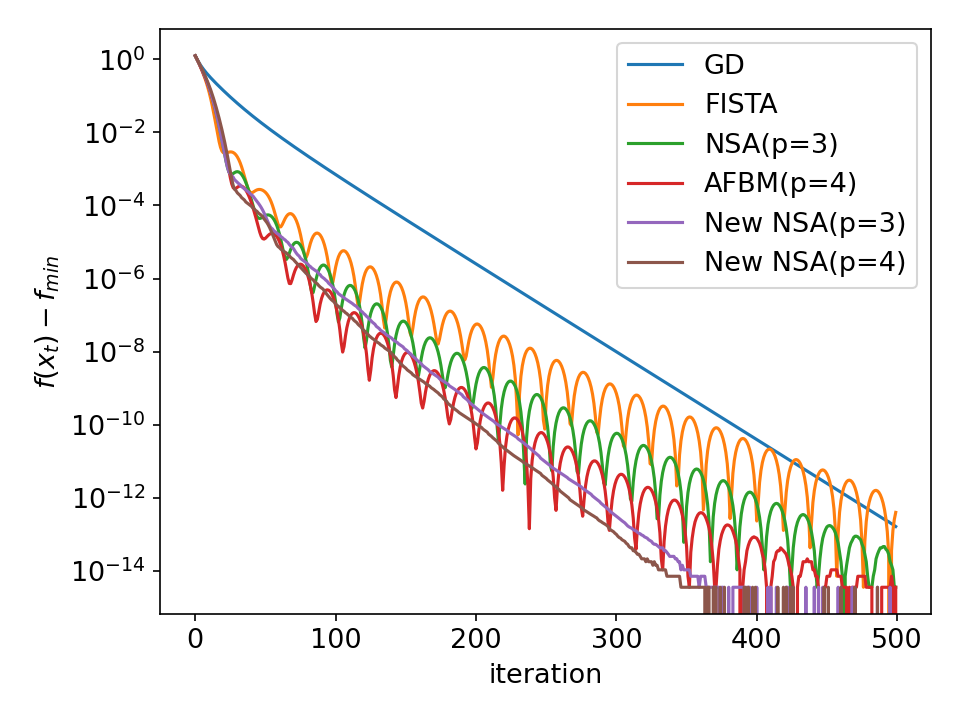} \label{fig4}} \\ 
     \subfloat[][$\min_{\x}\frac{1}{2} \|A \x-\b\|^2+\lambda\|\x\|_2^2 $]{\includegraphics[width = 0.45\linewidth]{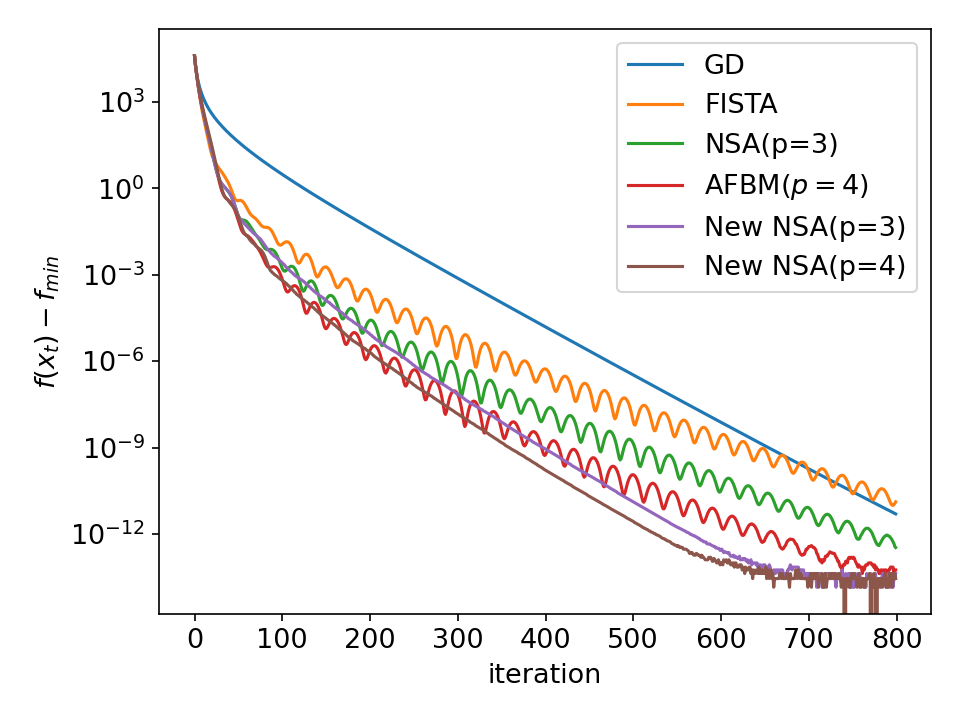} \label{fig5}} \hfill
     \subfloat[][$ \min_{\x} \frac{1}{2}\|X_{ob}-A_{ob}\|^2+\lambda \|X\|_* $]{\includegraphics[width = 0.45\linewidth]{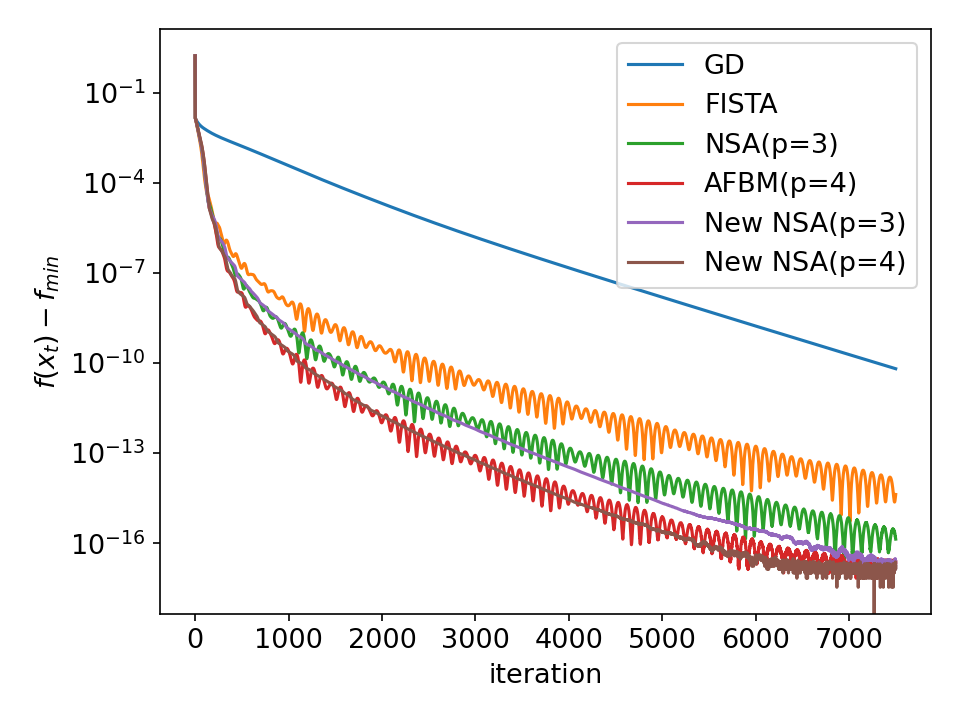} \label{fig6}} 
     \caption{Comparison of Algorithm \ref{alg:nsa-comp} (referred to as `New NSA') with benchmark methods. In the graph legend, GD represents Gradient Descent; NSA denotes the original acceleration method by Nesterov and Spokoiny \cite{2017Random}; and FISTA and AFBM methods were introduced in \cite{2009A} and \cite{nesterov1983method,Nesterov2013}, respectively. In this context, $p$ is used to indicate the damping factor. As the graphs for NSA and AFBM are closely aligned when using the same damping factor, we present only the NSA method with a damping factor of $p = 3$ and the AFBM method with a damping factor of $p = 4$. \label{fig:two} }
\end{figure}




\begin{enumerate}[label=\textbf{Figure \ref{fig:three}\alph*:}, leftmargin=2\parindent, align=left] 
    \item 
    We train
    a simple feedforward neural network with one hidden layer. The input data is passed through the hidden layer, where it is transformed using a sigmoid activation function. The result is then forwarded to the output layer, where a softmax activation function converts it into a probability distribution for classification. The network uses weights and biases to adjust the output at each layer. 
    We set $\eta=0.12$ for all algorithms. 
    We populate the origin parameter matrix with random samples from a uniform distribution over $[0, 1)$. The $y$-axis labels the training loss of the learning process. 
    \item 
    The parameters and settings are same as the experiment above. The only difference is that we change the horizontal axis from iteration to time, which means that we can compare the convergence speed of different algorithms within the same computation time.
\end{enumerate}

\begin{figure}[H]
     \centering
     \subfloat[][
     ]{\includegraphics[width = 0.45\linewidth]{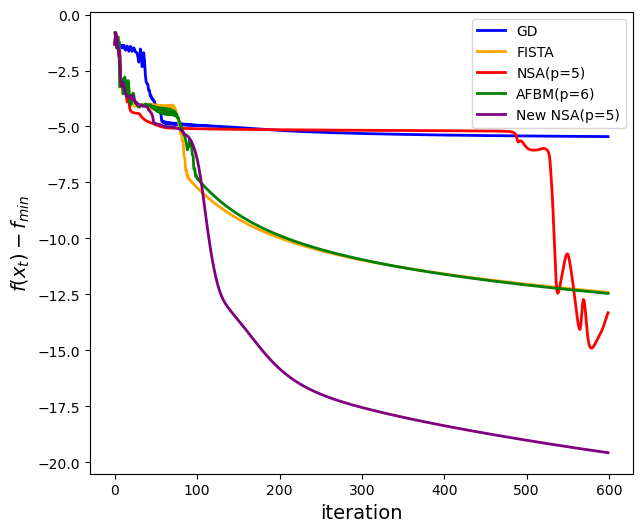} \label{fig1-new}} \hfill
     \centering
     \subfloat[][
     ]{\includegraphics[width = 0.45\linewidth]{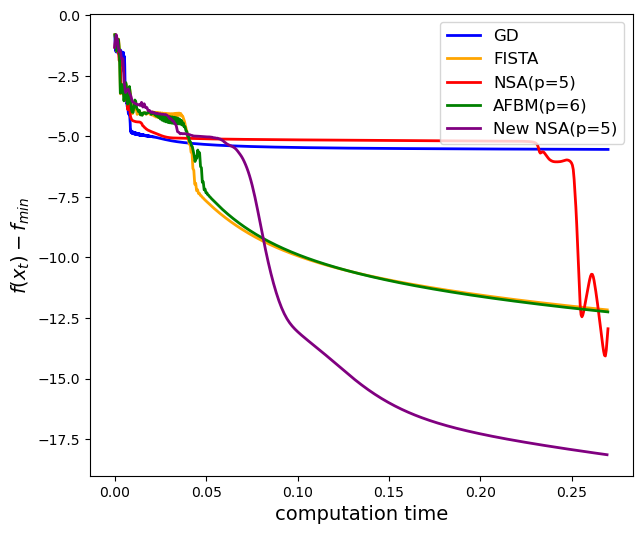} \label{fig2-new}} \hfill
      \caption{Comparison of Algorithm \ref{alg:nsa-comp} (referred to as `New NSA') with benchmark methods \if\highlight1 \color{red} \fi{($y$-axis in logarithmic scale)}. \color{black} In the graph legend, GD represents Gradient Descent; NSA denotes the original acceleration method by Nesterov and Spokoiny \cite{2017Random}; and FISTA and AFBM methods were introduced in \cite{2009A} and \cite{nesterov1983method,Nesterov2013}, respectively. In this context, $p$ is used to indicate the damping factor. We present only the NSA method with a damping factor of $p = 5$ and the AFBM method with a damping factor of $p = 6$. \label{fig:three} }
\end{figure}

\section{Conclusion}
\label{sec:conc}

This paper studies the NSA algorithm -- a variant of an acceleration method of Nesterov and Spokoiny. NSA converges at rate $o(k^{-2})$ for general smooth convex programs. To our knowledge, this work is the first to establish an $ o(k^{-2}) $ convergence rate for the function value and an $ o(\frac{1}{k^{3} \log k }) $ rate for the squared gradient norm, all while maintaining strictly monotonic descent of function value. 

This work also presents a comprehensive study of the NSA algorithm, providing  a zeroth-order variant capable of handling inexact gradients, a extension that handles convex composite problems, and a continuous-time analysis of its dynamics. 

\section*{Declaration of Generative AI and AI-assisted technologies in the writing process}

During the preparation of this work the authors used ChatGPT (Chat-4o-mini) in order to improve readability and polish writing. After using this tool/service, the authors reviewed and edited the content as needed and take full responsibility for the content of the publication. 

\bibliographystyle{apalike} 
\bibliography{bibliography}


\appendix

\section{Standard Conventions for Smooth Convex Optimization}


To be self-contained, we briefly review some common conventions for smooth convex optimization. 
In his book \cite{nest2018}, Nesterov's discussion on smooth convex optimization focuses on a class of functions $\mathscr{F}$ such that, for any $f \in \mathscr{F}$, $ \nabla f (\x^*) = 0 $ implies that $ \x^* $ is a global minimizer of $f$. In addition, $\mathscr{F}$ contains linear functions, and is closed under addition and nonnegative scaling. 

Following Nesterov's convention \cite{nest2018}, we write $ f \in \mathscr{F}_L^{h,l} (\R^n) $ ($l \le h$) if $f \in \mathscr{F}$, and in addition, 
\begin{itemize}
    \item $f$ is convex; 
    \item $f$ is $h$-times continuously differentiable; 
    \item The $ l $-th derivative of $f$ is $L$-Lipschitz: There exists $L \in (0,\infty)$ such that 
    \begin{align*}
        \| \partial^l f (\x) - \partial^l f ( \x' ) \| \le L \| \x - \x' \|, \quad \forall \x , \x' \in \R^n . 
    \end{align*}
\end{itemize}

\section{Proofs of Theorem \ref{thm:ge-b} and Corollary \ref{cor}} 
\label{sec:est}

In this section, we provide details of deferred proofs. 
Firstly, we provide a proof of Theorem \ref{thm:ge-b}. Together with Theorem \ref{thm:ge-a}, this shows that 
one can construct gradient oracles that satisfies properties \textbf{(H1)} and \textbf{(H2)} using zeroth-order information of the function. 
\begin{proof}[Proof of Theorem \ref{thm:ge-b}]
     $\mathcal{G} (f, \x , \epsilon;\xi)$ satisfies \textbf{(H2)} if the following inequality below is true almost surely
     \begin{equation}
         \frac{\lvert\langle \nabla f(\x)-\mathcal{G}(f, \x , \epsilon;\xi),\mathcal{G}(f, \x, \epsilon;\xi)\rangle\rvert}{\langle\mathcal{G}(f, \x, \epsilon;\xi),\mathcal{G}(f, \x, \epsilon;\xi)\rangle}\leq\frac{1}{4} . \nonumber 
     \end{equation}
     Using Schwarz inequality, we have
     \begin{equation}\label{thm2-le1}
      \frac{\lvert\langle \nabla f(\x)-\mathcal{G}(f, \x, \epsilon;\xi),\mathcal{G}(f, \x, \epsilon;\xi)\rangle\rvert}{\langle\mathcal{G}(f, \x, \epsilon;\xi),\mathcal{G}(f, \x, \epsilon;\xi)\rangle}\leq
      \frac{\lVert \nabla f(\x)- \mathcal{G}(f, \x, \epsilon;\xi)\rVert}{\lVert \mathcal{G}(f, \x, \epsilon;\xi)\rVert}
     \end{equation}
     Furthermore, the triangle inequality gives that 
     \begin{equation}
     \frac{\lVert \nabla f(\x)- \mathcal{G}(f, \x, \epsilon;\xi)\rVert}{\lVert \mathcal{G}(f, \x, \epsilon;\xi)\rVert}
     \leq
     \frac{\lVert \nabla f(\x)- \E_\xi\mathcal{G}(f, \x, \epsilon;\xi)\rVert+\lVert  \E_\xi\mathcal{G}(f, \x, \epsilon;\xi)-  \mathcal{G}(f, \x, \epsilon;\xi)\rVert}{\lVert \mathcal{G}(f, \x, \epsilon;\xi)\rVert} \nonumber 
     \end{equation}
     \textbf{Approach 1:}
     Select $\mathcal{G}(f, \x, \epsilon;\xi)$ as follows:
     \begin{equation*}
        \mathcal{G}(f, \x, \epsilon;\xi)= \frac{1}{2\epsilon }\sum_{k=1}^{n}\[ f(\x+\epsilon\mathbf{e_i} )-f(\x-\epsilon \mathbf{e_i})\]\mathbf{e_i}, \quad \forall \x \in \R^n
     \end{equation*}
     where $\mathbf{e_i}$ is the $i$-th unit vector. Then we use the Taylor formula to get
     \begin{equation}\label{thm2-le2} 
         \mathcal{G}(f, \x, \epsilon;\xi)=\nabla f(\x)+O(\epsilon \mathbf{1}), \quad 
         \lVert \mathcal{G}(f, \x, \epsilon;\xi)\rVert ^2=
         \lVert \nabla f(\x)\rVert^2+O(\epsilon), 
     \end{equation}
     where $\mathbf{1}$ denotes the all-one vector. 
     Therefore, there exists $M_1>0,M_2>0$,
     \begin{equation*}
         \lVert\mathcal{G}(f, \x, \epsilon;\xi)-\nabla f(\x)\rVert\leq M_1\epsilon, \quad 
         -M_2\epsilon\leq\lVert \mathcal{G}(f, \x, \epsilon;\xi)\rVert ^2-
         \lVert \nabla f(\x)\rVert^2 \leq M_2\epsilon
     \end{equation*}
     We choose $\epsilon$ so that 
     \begin{align}
         \epsilon \le \textbf{min}\{3(4M_2)^{-1}\lVert \nabla f(\x)\rVert^2,
     (8M_1)^{-1}\lVert \nabla f(\x)\rVert\} . \label{eq:def-eps} 
     \end{align}
     then we have
     \begin{equation}
         \frac{\lVert \nabla f(\x)- \mathcal{G}(f, \x, \epsilon;\xi)\rVert}{\lVert \mathcal{G}(f, \x, \epsilon;\xi)\rVert}\leq
         \frac{M_1\epsilon}{\sqrt{\lVert \nabla f(\x)\rVert^2-M_2\epsilon}}\leq
         \frac{2M_1}{\lVert \nabla f(\x)\rVert}\epsilon\leq \frac{1}{4} \nonumber 
      \end{equation}
      Therefore, from (\ref{thm2-le1}), we conclude that $\mathcal{G}(f, \x, \epsilon;\xi)$ above satisfies \textbf{(H2)}.\\
      \textbf{Approach 2:}
      Select $\mathcal{G}(f, \x, \epsilon;\xi)$ as follows:
     \begin{equation*}
        \mathcal{G}(f, \x, \epsilon;\xi)= 
        \frac{n}{2 \epsilon k} \sum_{i=1}^k \left[ f ( \x + \epsilon \v_{i} ) - f ( \x - \epsilon \v_{i}) \right] \v_{i}, \quad \forall \x \in \R^n
     \end{equation*}
     where $ [\v_{1}, \v_{2}, \cdots, \v_{k}] = \V \in$   $  \{ \mathbf{X} \in \R^{n \times k} : \mathbf{X}^\top \mathbf{X} = \mathbf{I}_k \} $ is uniformly sampled.\\
     Select $k=n$, then $ [\v_{1}, \v_{2}, \cdots, \v_{n}] $ is an orthonormal bases for $\R^{n}$, and we have
    \begin{equation*}
        \mathcal{G}(f, \x, \epsilon;\xi)= 
        \frac{1}{2 \epsilon } \sum_{i=1}^n \left[f ( \x + \epsilon \v_{i} ) - f ( \x - \epsilon \v_{i}) \right] \v_{i}, \quad \forall \x \in \R^n.
     \end{equation*}
     Let $v_i^k$ represents the $k$-th element of $\mathbf{v_i}$ . Using the Taylor formula,we have
      \begin{equation*}
        \mathcal{G} (f, \x , \epsilon;\xi)= 
        \(\sum_{i=1}^n\sum_{k=1}^{n}\frac{\partial f}{\partial \x_k}v_i^kv_i^p \)_{p=1}^n+O(\epsilon\mathbf{1}) =
        \(\sum_{k=1}^{n}\frac{\partial f}{\partial \x_k}\[\sum_{i=1}^n v_i^kv_i^p\] \)_{p=1}^{n}+O(\epsilon\mathbf{1})
     \end{equation*}
      $ [\v_{1}, \v_{2}, \cdots, \v_{k}] = \V \in$   $  \{ \mathbf{X} \in \R^{n \times n} : \mathbf{X}^\top \mathbf{X} = \mathbf{I}_n \} $ So it is obvious that $\mathbf{V}$ is an orthogonal matrix and $\sum_{i=1}^n v_i^kv_i^p$ is the inner product between the $k$-th column and the $p$-th column. According to the proposition of the orthogonal matrix, we know $\sum_{i=1}^n v_i^kv_i^p=\delta_{kp}$ , where $\delta_{kp}$ is the Kronecker delta. So we can simplify the formula above as
      \begin{equation}
        \mathcal{G}(f, \x , \epsilon;\xi)=\nabla f(\x)+o(\epsilon\mathbf{1}) \nonumber 
       \end{equation}
      Therefore we get (\ref{thm2-le2}). Repeat the process in Approach 1, we can conclude that $\mathcal{G}(f, \x, \epsilon;\xi)$ in Approach 2 satisfies \textbf{(H2)}.
\end{proof}

\if\highlight1
\color{red}
\fi 

\begin{proof}[Proof of Corollary \ref{cor}] 
    To prove Corollary \ref{cor}, it is sufficient to show that when (\ref{eq:cor-1}) is false, conditions (a) and (b) in Theorem \ref{thm:main-inexact} are satisfied. 
    Suppose (\ref{eq:cor-1}) is false. By our choice of $\epsilon_k$, we know for some $\beta > 1$, the sequence $ \{ \beta^k \| \nabla f (\x_k) \| \}_k $ is bounded away from 0. Therefore, a proper choice of $\beta$ guarantees that (\ref{eq:def-eps}) holds for all $k$ larger than a constant. 
    Therefore, with our choice of $\epsilon_k$, \textbf{(H2)} is satisfied (for all $k$ larger than a constant). Thus condition (a) in Theorem \ref{thm:main-inexact} holds true since \textbf{(H1)} is also satisfied. 

    Also, with this choice of $\epsilon_k$, condition (b) in Theorem \ref{thm:main-inexact} is also satisfied. This concludes the proof. 
\end{proof} 

\color{black}

\section{Additional Experiments}
\subsection{Zeroth-order version of NSA} 
We compare the NSA-zero algorithm (\ref{alg:nsa-inexact}) with classic gradient descent and the Nesterov Accelerated Gradient (NAG) \cite{nesterov1983method}. 
    In all experiments, we pick the same values of $ p $ and $\eta$ for both NSA-zero and NAG. The start point $\x_0$  is also same for both NSA-zero and NAG. We choose the gradient estimation $\mathcal{G}$ as
\begin{equation*}
\g_t(\x) = \frac{1}{2\epsilon_t}\sum_{i=1}^{n} \left[f(\x+\epsilon_t\v_i)-f(\x-\epsilon_t\v_i)\right]\cdot \v_i.
\end{equation*} 

\begin{enumerate}[label=\textbf{Figure \ref{fig:app}\alph*:}, leftmargin=2\parindent, align=left] 
    \item The cost function of least square regression
    \begin{equation*}
        f(\x) = \frac{1}{2}\|A \x- \b \|^2,
    \end{equation*}
    where $A$ is a $200 \times 100$ random matrix   with entries sampled as $i.i.d.$ standard Gaussian, $ \: \mathcal{N}(0,1)$ and $\b$ is a $200 \times 1$ standard Gaussian vector. We set $\eta=0.0005$ for all algorithms.
    \item The cost function of logistics regression
    \begin{equation*}
    f(\x)=\sum_{i=1}^{200}\left[- y_i A_i^\top \x+ \log(1+e^{A_i^\top \x})\right],
    \end{equation*} 
    where $A =\left(A_1 , ..., A_{200}\right)^\top$ is a $200\times 5$ random matrix sampled from $i.i.d.\:  \mathcal{N}(0,1)$ and $\y = (y_1, \cdots, y_{200})^\top$ is a $200\times 1$ random vector of labels sampled from $i.i.d.\: Ber(1,0.5)$. We set $\eta=0.0005$ for all algorithms.
    \item The cost function of  lasso regression
    \begin{equation*}
        f(\x) = \frac{1}{2} \|A \x-\b\|^2+\lambda \|\x\|_{1},
    \end{equation*}
     where $A$ is a $200 \times 100$ random matrix   with entries sampled as $i.i.d.$ standard Gaussian, $ \: \mathcal{N}(0,1)$ and $\b$ is a $200 \times 1$ standard Gaussian vector. We set $\eta=0.0005$ for all algorithms. 
    \item  The \texttt{log\_exp\_sum} function
    \begin{equation*}
    f(\x) = \rho \log \left[ \sum_{i=1}^{n}e^{(A_i^\top \x-b_i)/\rho}\right],
    \end{equation*}
    where $A = \left[ A_1,...A_{100}\right]^T$ is a $100\times 20$ random matrix sampled from $i.i.d. \:\mathcal{N}(0,1)$ and $\b = (b_1, b_2, \cdots, b_{100})^\top$ is a $100\times 1$ random vector  sampled from $i.i.d.\: \mathcal{N}(0,1)$. Let $ \rho = 5$. We set $\eta=0.02$ for all algorithms.
    \item  The cost function of  ridge regression
    \begin{equation*}
        f(\x) = \frac{1}{2} \|A\x-\b\|^2+\lambda \|\x\|_{2}^2,
    \end{equation*}
     where $A$ is a $200 \times 100$ random matrix   with entries sampled as $i.i.d.$ standard Gaussian, $ \: \mathcal{N}(0,1)$ and $\b$ is a $200 \times 1$ standard Gaussian vector. We set $\eta=0.0005$ for all algorithms. 
    \item  The cost function of matrix completion
    \begin{equation*}
    f(X)=\frac{1}{2}\|X_{ob}-A_{ob}\|^2+\lambda\|X\|_{*},
    \end{equation*}
    where  $A_{ob}$ includes the available elements in matrix $A$ and only 20\% of the elements in $A$ can be observed. The position of these elements is randomly chosen and the original matrix $A$ is a rank-5  $10\times 8$ matrix with the eigenvalues \textemdash$\left[1,2,3,4,5\right]$. To generate $A$, we  generate $U,S$ and $V$ where $S=diag\{1,2,3,4,5\}$ and $U,V^\top$ are rank-3 $50\times7$ matrices generated from $i.i.d $ $\: \mathrm{Uniform}(0,1)$. Then we calculate $A$ by $A = USV$, and set $\lambda=0.005$. We set $\eta=0.05$ for all algorithms.
\end{enumerate} 

\begin{figure}[H]
     \centering
     \subfloat[][$\min_{\x} \frac{1}{2}\|A \x - \b\|^2$]{\includegraphics[width = 0.45\linewidth]{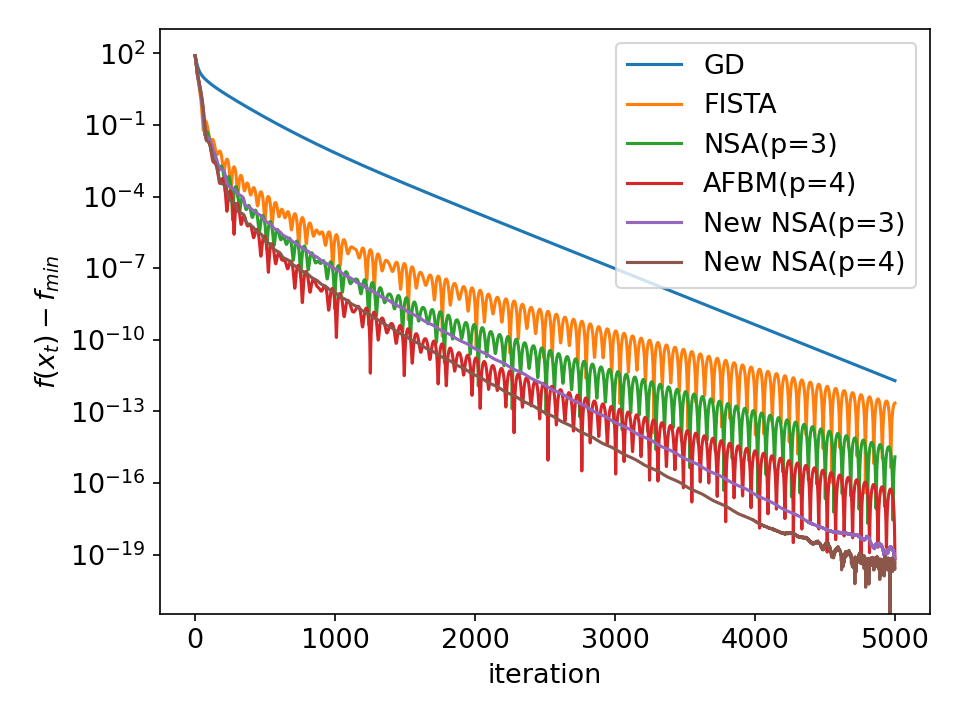} \label{fig7}} \hfill
     \subfloat[][$ \min_{\x} \sum_{i=1}^{n}-y_iA_i^\top \x+ \log(1+e^{A_i^\top \x}) $]{\includegraphics[width = 0.45\linewidth]{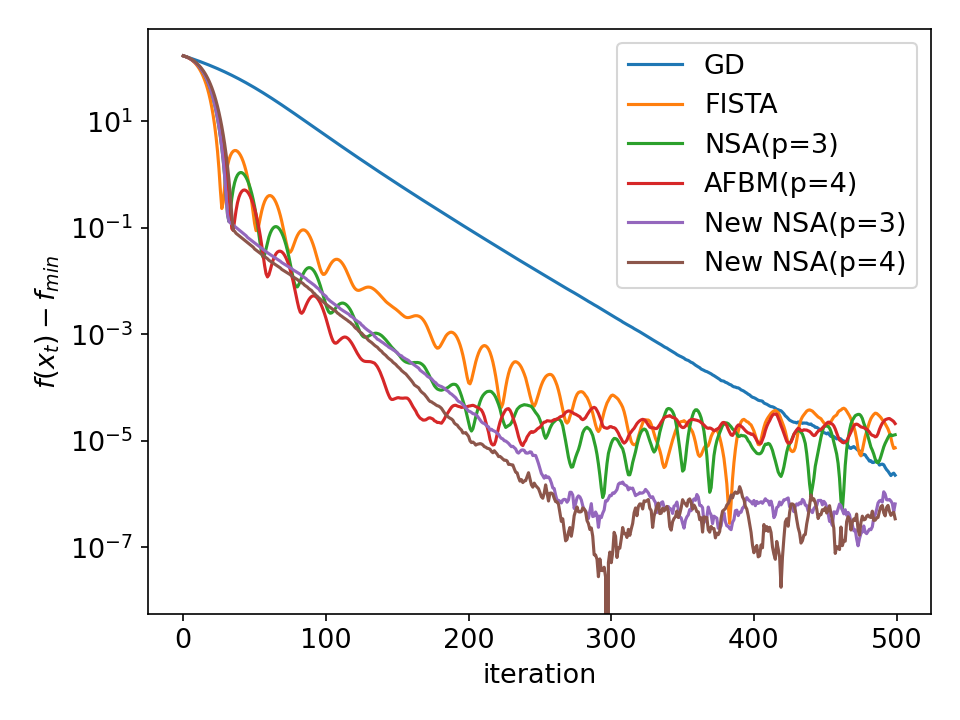} \label{fig8}} \\ 
     \subfloat[][$ \min_{\x} \frac{1}{2} \|A \x- \b\|^2+\lambda \|x\|_{1} $]{\includegraphics[width = 0.45\linewidth]{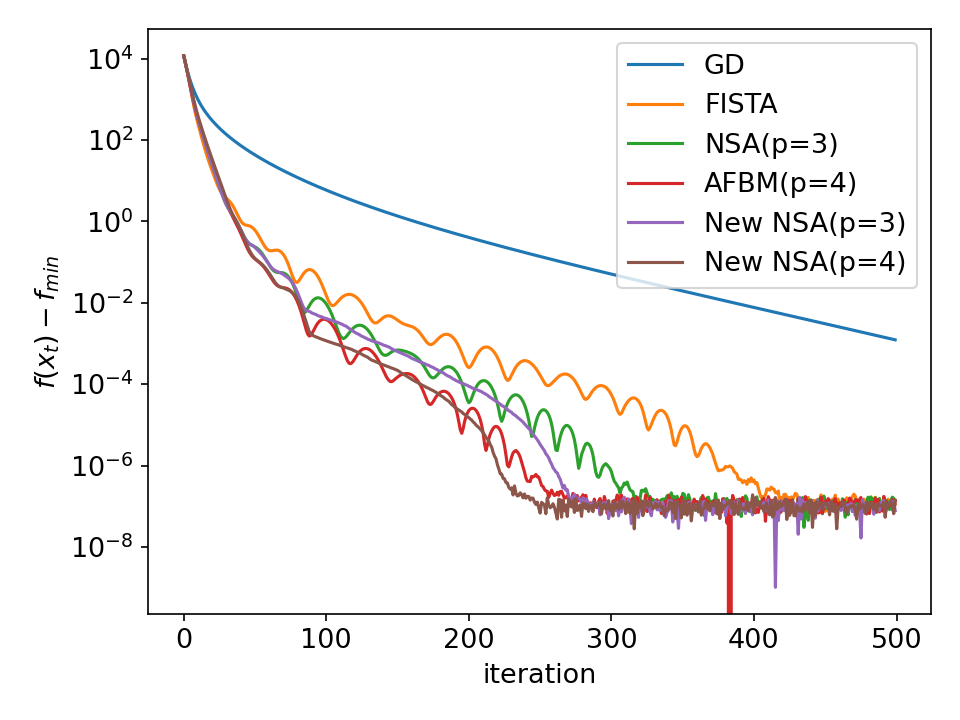} \label{fig9} } \hfill
     \subfloat[][$ \min_{\x} \rho \log \[ \sum_{i=1}^{n} \exp(A_i^\top \x-b_i)/\rho \] $]{\includegraphics[width = 0.45\linewidth]{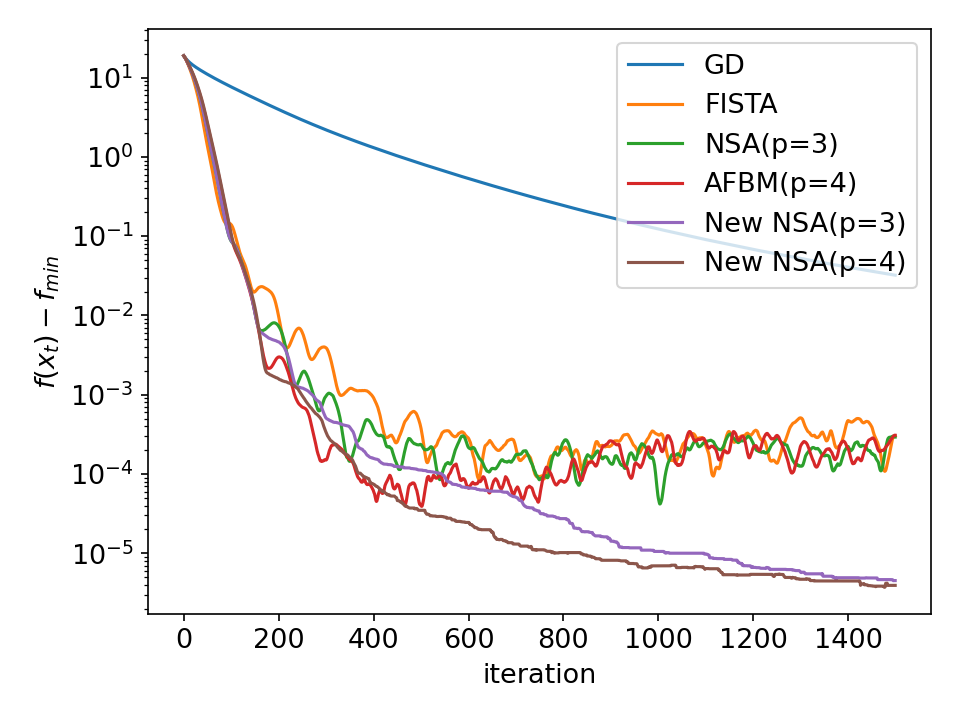} \label{fig10}} \\ 
     \subfloat[][$\min_{\x}\frac{1}{2} \|A \x-\b\|^2+\lambda\|\x\|_2^2 $]{\includegraphics[width = 0.45\linewidth]{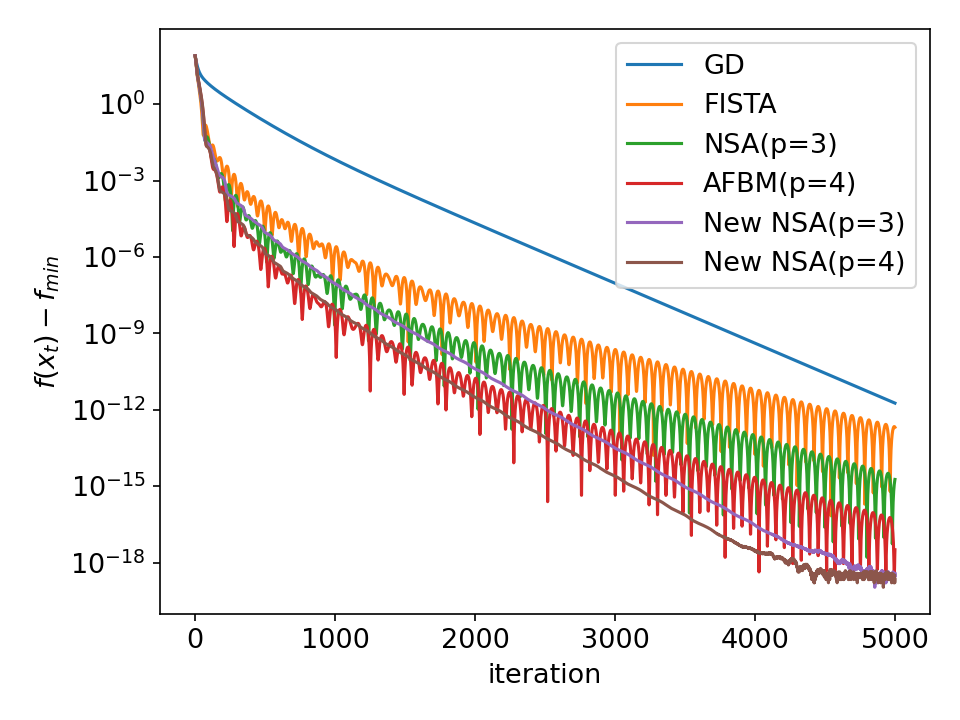} \label{fig11}} \hfill
     \subfloat[][$ \min_{\x} \frac{1}{2}\|X_{ob}-A_{ob}\|^2+\lambda \|X\|_* $]{\includegraphics[width = 0.45\linewidth]{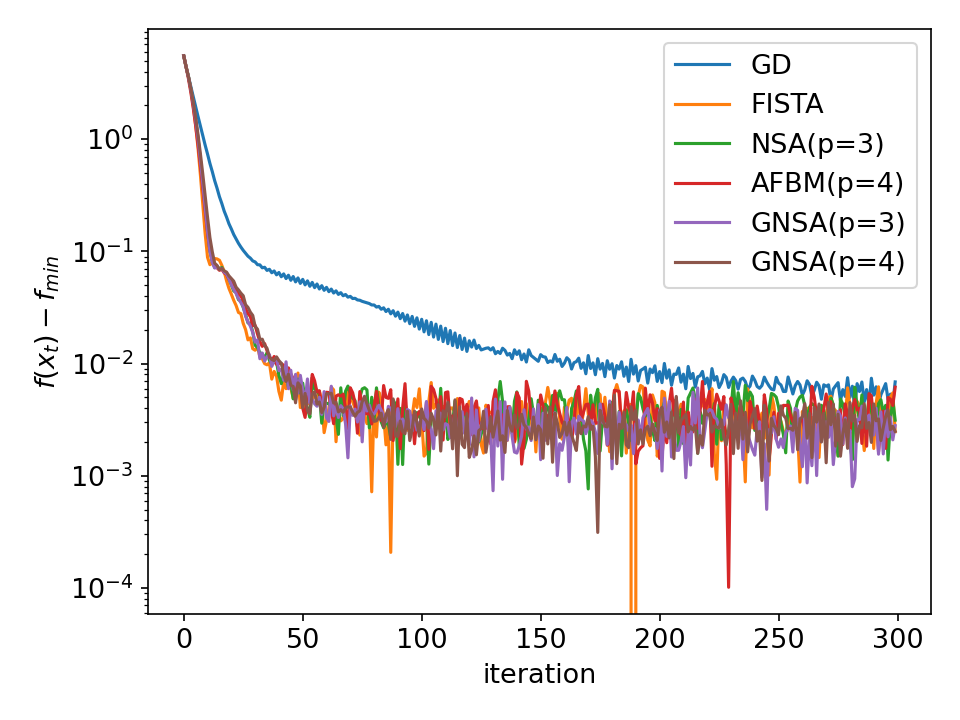} \label{fig12}} 
     \caption{Comparison of Algorithm \ref{alg:nsa-inexact} (referred to as `New NSA') with benchmark methods. In the graph legend, GD represents Gradient Descent; NSA denotes the original acceleration method by Nesterov and Spokoiny \cite{2017Random}; and FISTA and AFBM methods were introduced in \cite{2009A} and \cite{nesterov1983method,Nesterov2013}, respectively. In this context, $p$ is used to indicate the damping factor. As the graphs for NSA and AFBM are closely aligned when using the same damping factor, we present only the NSA method with a damping factor of $p = 3$ and the AFBM method with a damping factor of $p = 4$. In this experiments, all methods use inexact gradient obtained from zeroth-order information. \label{fig:app} }
\end{figure}




\end{document}